\documentclass[11pt]{amsart}

\usepackage{enumitem}
\usepackage{amsmath}	
\usepackage{amssymb}	
\usepackage{amsthm}
\usepackage[utf8]{inputenc}
\usepackage{url}
\usepackage{mathtools}
\usepackage[english]{babel}
\usepackage{float}
\usepackage{amsfonts}
\usepackage{verbatim}
\usepackage[T1]{fontenc} 	
\usepackage{amscd}
\usepackage{calc}
\usepackage[hidelinks]{hyperref}

\usepackage{graphicx}

\usepackage{tikz}
\usetikzlibrary{arrows.meta}

\newtheorem{thm}{Theorem}[section]
\newtheorem{cor}[thm]{Corollary}
\newtheorem{lem}[thm]{Lemma}

\theoremstyle{definition}
\newtheorem{defn}[thm]{Definition}
\newtheorem{exam}[thm]{Example}
\newtheorem{rem}[thm]{Remark}

\numberwithin{equation}{section}

\newcommand{\norm}[1]{\left\Vert#1\right\Vert}
\newcommand{\abs}[1]{\left\vert#1\right\vert}

\newcommand{\B}{\mathcal{B}}

\newcommand*\HH{\mathcal{H}} 

\newcommand{\field}[1]{\mathbb{#1}}

\newcommand{\C}{\field{C}}

\newcommand{\D}{\field{D}}

\newcommand{\HOLO}{\mathcal H}

\newcommand*\closed[1]{\overline{#1}}

\newcommand*\LNOP[1]{L_{#1}}

\newcommand{\BOAFW}{u} 
\newcommand{\BOAF}{H^\infty_\BOAFW}
\newcommand{\BOAFS}{H^0_\BOAFW}

\catcode`,\active

\catcode`\,12

\allowdisplaybreaks

\begin{document}
\title[Exact essential norm ...]{Exact essential norm of generalized  Hilbert matrix operators on classical analytic function spaces}

\author[Lindstr{\"o}m] {M. Lindstr\"om}\address{Mikael Lindstr{\"o}m. Department of Mathematics, \AA bo Akademi University. FI-20500 \AA bo, Finland. \emph{e}.mail: mikael.lindstrom@abo.fi} 
\author[Miihkinen] {S.Miikinen}\address{Santeri Miihkinen. Department of Mathematics and Computer Science, Karlstad University. SE-651 88 Karlstad, Sweden. \emph{e}.mail: miihkine@gmail.com}
\author[Norrbo] {D. Norrbo}\address{David Norrbo. Department of Mathematics, \AA bo Akademi University. FI-20500 \AA bo, Finland.
\emph{e}.mail: david.norrbo@abo.fi}
\subjclass[2010]{Primary 47B38, Secondary 30H20.}
\keywords{Hilbert matrix operator, essential norm, weighted composition operator, weighted Bergman spaces,  weighted Banach spaces of analytic functions}

\begin{abstract}
We compute the exact value of the essential norm of a generalized Hilbert matrix operator acting on weighted Bergman spaces $A^p_v$ and weighted Banach spaces $H^\infty_v$ of analytic functions, where $v$ is a general radial weight.  
In particular, we obtain the exact value of the essential norm of the classical Hilbert matrix operator on standard weighted Bergman spaces $A^p_\alpha$ for $p>2+\alpha, \, \alpha \ge 0,$ and on Korenblum spaces $H^\infty_\alpha$ for $0 < \alpha < 1.$ We also cover the Hardy space $H^p, \,  1 < p < \infty,$ case. In the weighted Bergman space case, the essential norm of the Hilbert matrix is equal to the conjectured value of its operator norm and similarly in the Hardy space case the essential norm and the operator norm coincide. We also compute the exact value of the norm of the Hilbert matrix on $H^\infty_{w_\alpha}$ with weights $w_\alpha(z)=(1-|z|)^\alpha$ for all $0 < \alpha < 1$. Also in this case, the values of the norm and essential norm coincide.

\end{abstract}

\maketitle
\section{\bf{Introduction}}

In recent years, a significant interest has arisen to compute the exact norm of the Hilbert matrix operator $\mathcal{H}$ on classical spaces of analytic functions on the open unit  disk, such as Hardy spaces, weighted Bergman spaces and the Korenblum spaces, see \cite{DJV}, \cite{BMBK}, \cite{ BK2}, \cite{BK3}, \cite{LMW} and \cite{LMW2}. A central tool in determining the norm of $\mathcal{H}$ on these spaces is an integral representation of $\mathcal{H}$ in terms of certain weighted composition operators established by Diamantopoulos and Siskakis  in \cite{DS}. We generalize this approach by considering an integral operator $I_K$ with a kernel $K$ satisfying some natural conditions and establishing its exact essential norm on weighted Bergman spaces and weighted Banach spaces of analytic functions. The Hilbert matrix $\mathcal{H}$ is a prime example of this generalized Hilbert matrix operator $I_K$. In consequence of this, our results for the exact essential norm of the operator $I_K$ produce the exact essential norm of the Hilbert matrix on the mentioned spaces.

A concise history of the previous work on computation of the norm of $\mathcal{H}$ acting on weighted Bergman spaces $A^p_\alpha$ is as follows. Diamantopoulos  \cite{D1} was the first one to study the boundedness  of $\mathcal  H$ on Bergman spaces and, in \cite{BMBK},  Bo\v{z}in and Karapetrovi\'{c}  proved the conjecture on the exact norm  of $\mathcal{H}$  on $A^p$. The study of boundedness of $\mathcal{H}$ on $A^p_\alpha$ was initiated and some partial results were obtained in \cite{GGPS}.  In \cite{BK2}, Karapetrovi\'{c} determined the  exact norm of $\mathcal{H}$ on $A^p_\alpha$ in the case $4\leq 2(2+\alpha) \leq p < \infty$. However, he obtained that the same lower bound holds for all $p> 2+\alpha> 1.$ Karapetrovi\'{c}  conjectured in \cite{BK2} that the upper bound for the norm of $\mathcal{H}$ is the same as above also in the case 
$1 < 2+\alpha < p < 2(2+\alpha)$. In \cite{LMW2} Lindstr\"om, Miihkinen and Wikman confirmed the conjecture in the positive for $2+\alpha +\sqrt{\alpha^2+\frac{7}{2}\alpha +3} \le p < 2(2+\alpha).$  Very recently Karapetrovic  \cite{BK3}  generalized the work of   \cite{LMW2} by showing  that the conjecture holds  for 
$2+\alpha +\sqrt{(2 + \alpha)^2-(\sqrt {2} - \frac{1}{2})(2 +\alpha)} \le p < 2(2+\alpha)$, which is an improvement. Regarding the Korenblum spaces $H^\infty_\alpha$, $0 < \alpha < 1$, the exact norm of $\mathcal{H}$ was computed for small values of $\alpha$ and an upper estimate was established for large values of $\alpha$ in \cite{LMW}. See also \cite{AMS}.

The aim of this paper is to find the exact value of the essential norm  of the Hilbert matrix operator on weighted  Bergman spaces $A^p_\alpha$ for all values of $p>2+\alpha, \, \alpha \ge 0,$ and on the Korenblum spaces $H^\infty_\alpha$ and $H^0_\alpha$ for  $ 0 < \alpha < 1$. We also cover the Hardy space $H^p, \, 1 < p < \infty,$ case. Our approach to this task is quite general. Namely, we consider an integral operator defined on weighted Bergman spaces $A^p_v$ and weighted Banach spaces $\BOAF$ of analytic functions. The kernel of this operator satisfies some general and natural conditions and the operator can be considered a generalized version of the classical Hilbert matrix operator through an integral representation via certain weighted composition operators. We establish the exact value of the essential norm of this integral operator acting on the mentioned spaces.

The organization of the article  is the following. After recalling the relevant analytic function spaces and notions, we  introduce the generalized Hilbert matrix integral operator $I_K$ in Section \ref{sec:Prelim}. Then we move on to present some important  lemmas in Section \ref{sec:SomeUsefulResults}. In Section \ref{sec:UpperBound} upper bounds for the essential norms of $I_K$ on weighted Bergman spaces $A^p_v$ and weighted Banach spaces of analytic functions $\BOAF$ are derived. Sections \ref{sec:LowerBoundOfIKOnAp} and \ref{sec:LowerBoundOfIKOnHInfty} contain derivations of lower bounds of essential norms of $I_K$ on weighted Bergman spaces $A^p_v$ and $\BOAF$ respectively. Section \ref{sec:Examples} presents central examples of our results. Namely, exact values of essential norms of the classical Hilbert matrix operator are derived in cases of the standard weighted Bergman spaces $A^p_\alpha$, the Korenblum spaces $H^\infty_\alpha$ (and $H^0_\alpha$) and the Hardy spaces $H^p$. Section \ref{sec:EssNormOfWCompOp}
 consists of results related to essential norm for a class of weighted composition operators. In Section \ref{sec:InterestingID} we derive a general formula describing the essential norm of the operator $I_K$ on spaces $A^p_\alpha$ and $H^\infty_\alpha$ as an integral average of essential norms of certain weighted composition operators. We determine the exact value of the norm of $\HH$ acting on $H^\infty_{w_\alpha}$ spaces with weights $w_\alpha(z) = (1-|z|)^\alpha$ and it is shown to coincide with the essential norm. Finally, we state that the pair $(A^p_v,A^q_v),\ q\geq p>1,$ has the weak maximizing property (see also \cite{AGPET} and \cite{PT}) and point out why this might be useful in determining the norm of $\HH$ on $A^p_\alpha, \, p-2>\alpha > -1$.

\section{Preliminaries}\label{sec:Prelim}

The space of analytic functions on the open unit disk  $\mathbb{D}$ in the complex plane $\mathbb{C}$ is denoted by $\mathcal{H}(\D)$. Every analytic self-map $\varphi \colon \D \to \D$ induces a \emph{composition operator} $C_{\varphi}f = f \circ \varphi$ on $\mathcal{H}(\mathbb{D})$. If furthermore $\psi \in \mathcal{H}(\mathbb{D})$, then one can define a \emph{weighted composition operator} $\psi C_{\varphi}(f) = \psi\cdot(f \circ \varphi)$. For general information on composition operators on classical spaces of analytic functions, the reader is referred to the excellent monographs by Cowen and MacCluer \cite{CM} and Shapiro \cite{S}.

In this paper, we consider three spaces, the \emph{weighted Bergman spaces}  $ A^p_v,\ 1<p<\infty$, the \emph{weighted Banach spaces of analytic functions} $\BOAF$ and the \emph{Hardy spaces} $H^p,\ 1<p<\infty$. 

Let $v:\D\to\mathbb R_{> 0}$ be a continuous function satisfying $v(z)=v(\abs{z}), \ z\in\D$. We define the weighted Bergman spaces $A^p_v$ as the set
\[
\left\{f\in \HOLO (\D) :  \int_{\D}  \abs{f(z)}^p v(z)   \frac{dA(z)}{\pi} < \infty \right\}
\]
equipped with its natural norm,
\[
\norm{f}_{A^p_v}: =  \left( \int_{\D}  \abs{f(z)}^p v(z)   \frac{dA(z)}{\pi}\right)^\frac{1}{p},
\]
where $dA=dxdy$ is the Lebesgue area measure. When $f$ is measurable, not necessarily analytic, we will use $\norm{f}_{L^p_v}$ for the norm expression above. We will, furthermore, assume that the weight function $v$ for a weighted Bergman space is normalized, that is,
\[
 \int_{\D} v(z) \frac{dA(z)}{\pi}=1.
\]

The notation $A^p_\alpha, \ \alpha>-1,$ is used to denote the space $A^p_v$ with $v(z)=(1+\alpha)(1-\abs{z}^2)^\alpha$ and $dA_\alpha(z)=(1+\alpha)(1-\abs{z}^2)^\alpha \frac{dA(z)}{\pi}$.

We also define the weighted Banach spaces of analytic functions $H^\infty_v$ as the set
\[
\{f\in \HOLO (\D) : \norm{f}_{H^\infty_v}: = \sup_{z\in\D} \abs{f(z)} v(z) < \infty \},
\]
where we have assumed that $v$ is normalized, that is, $\sup_{z\in\D} {v(z)} = 1$. For the same norm expression we use the notation $\norm{f}_{L^\infty_v}$, when $f$ is measurable but not necessarily analytic. If $v\equiv 1$ we use the shorter notations $H^\infty := H^\infty_v$ and $ L^\infty := L^\infty_v $. We will also use the notation $\norm{f}_\infty := \norm{f}_{H^\infty}$ in Section \ref{sec:EssNormOfWCompOp}.

We now introduce a subclass of weight functions, namely, those $v$ that tend to zero on the boundary of $\D$. An arbitrary weight of this class will be denoted $\BOAFW$. The notation $H^\infty_\alpha, \ \alpha> 0,$ is used to denote the space $\BOAF$ with $\BOAFW(z)=(1-\abs{z}^2)^\alpha$. We also define a subspace $\BOAFS$ consisting of functions $f\in \BOAF$ such that
\[
\lim_{\abs{z}\to 1} \abs{f(z)} \BOAFW(z)=0.
\]

It will be assumed that for $X=A^p_v$ or $X=\BOAF$ the evaluation maps, $\delta_z$, $f\mapsto f(z), \ z\in\D,$ belong to $X^*$ and that $\int_0^1\norm{\delta_s}_{X^*} ds<\infty$. The last demand could be compared to the weaker demand $X\subset A^1_{[0,1)}$, where $A^1_{[0,1)}$ can be found in \cite{GGPS}. Moreover, the space
$A^p_v$ is assumed to be reflexive if $p>1$. Notice that $A^p_v$ is separable when the weight function is radial. 

The Hardy space $H^p$ for $1 \leq  p < \infty$ consists of all functions $f$ analytic in the unit disk such that
\begin{equation*}
\norm{f}_{H^p}^{p} := \sup_{0 \leq r < 1}\frac{1}{2\pi}\!\!\hspace*{0.1mm}\int_{0}^{2\pi}{|f(re^{i\theta})|^p d\theta} < \infty.
\end{equation*}

By ${\mathcal{ L}}(X,Y)$ we denote the space of all bounded linear operators  between   Banach spaces $X$ and $Y$. The \emph{essential norm} of $T\in {\mathcal{ L}}(X,Y)$  is defined to be the distance to the compact operators, that is
\begin{equation*}
\norm{T}_{e,X\to Y} = \inf\big\{\norm{T-L}_{X\to Y}:  \, L\in {\mathcal{ L}}(X,Y) \ \text{is compact}\big\}.
\end{equation*}
Notice that  $T\in {\mathcal{ L}}(X,Y)$ is compact if and only if $\norm{T}_{e,X\to Y}=0$. When $X=Y$, we write ${\mathcal{ L}}(X,X)={\mathcal{ L}}(X).$

\begin{defn}
Let $K:\D \times \D \to \C$ be analytic and $K(z,\cdot)\in H^\infty(\D)$ for every $z\in\D$. 
When we consider the operator
\[
I_K (f)(z) := \int_0^1 f(x) K(z,x)  dx,
\]
it will be assumed that it is a bounded operator, $X \to X$, where $X= A^p_v$ or $X= \BOAF$.
\end{defn}
Let
\[
T_t(z) :=  K(z,x_z(t))  x_z'(t), 
\]
where $x_z(t) = \frac{t}{(t-1)z +1}$. We also introduce $\phi_t (z) = x_z(t)$ and $w_t(z) = \frac{\phi_t (z) }{t}$. We may now write the operator $I_K$ as
\[
I_K (f)(z) = \int_0^1 T_t(z) f(\phi_t(z)) dt,
\]
using similar methods as in \cite{D1}, namely   changing the path of integration by  $x_z(t)$, $0\le t \le 1$. Since there are some differences in the dominating function, we provide the reader with some details. Let $X=A^p_v$ or $X=\BOAF$, $z\in\D$, and let $f\in X$. For $r\in(0,1)$ we define
\[
t(s) = t_{z,r}(s) = \frac{rs}{r(s-1)z+1}= \phi_{r,s}(z).
\]
We will use the fact that $t_{z,r}(s) = r\phi_s(rz)$. It follows that $t_{z,r}([0,1]) \subset\closed{B(0,r)}$, which shows that $f(w)K(z,w)$ as a function of $w$ is analytic in an open subset of $\D$ containing the closed curve $t_{z,r}([0,1])\cup [0,r]$. Therefore, it follows from the residue theorem that
\[
\int_{[0,r]} f(t) K(z,t) dt = \int_{t_{z,r}[0,1]} f(t) K(z,t) dt  = \int_0^1 f(t(s))K(z,t(s)) t'(s) ds.    
\]
Given a dominating integrable function to $ f(t(s))K(z,t(s)) t'(s)$ independent of $r$, the statement follows using the dominated convergence theorem $(r\to 1)$ because $f$ and $K(z,\cdot)$ are continuous at $t_{z,1}(s)$ for a fixed $0<s<1$. Since $t'(s) = rx'_{rz}(s) =\frac{r(1-rz)}{1-(1-s)rz} $ we have 
$\abs{K(z,t(s)) t'(s)}\leq   \sup_{w\in\D}\abs{K(z,w) }   \frac{ 2      }{  (1-\abs{z})^2   }<\infty $. All that is left to prove is that $ \int_0^1\sup_{0<r<1} \abs{f(t_{z,r}(s))}ds $ is finite. Let $g:=\frac{f}{\norm{f}_X}$. Since the weight is radial the evaluation maps, $\norm{\delta_z}_{X^*}$, are rotational symmetric, that is, $\norm{ \delta_z}_{X^*} = \norm{ \delta_{\abs{z}}  }_{X^*}$. From the maximum modulus principle we have $\norm{  \delta_{r_1}  }_{X^*}\leq \norm{  \delta_{r_2} }_{X^*}, \ 0\leq r_1\leq r_2$. Applying these properties we obtain
\[
\abs{g(t_{z,r}(s))} \leq \norm{\delta_ {\abs{t_{z,r}(s)}}  }_{X^*}\leq \norm{\delta_ {t_{\abs{z},r}(s)}  }_{X^*} \leq\norm{\delta_ {t_{\abs{z},1}(s)}  }_{X^*}.
\]
Moreover, $t_{\abs{z},1}'(s) = \frac{1-\abs{z}}{(1-(1-s)\abs{z})^2}\geq 1-\abs{z}$. Hence the line $h(s):=\abs{z}+(1-\abs{z})s$ dominates $t_{\abs{z},1}(s) $ and by the same arguments applied above we obtain $\norm{\delta_{t_{\abs{z},1}(s)}  }_{X^*} \leq \norm{\delta_{h(s)} }_{X^*}$. Since $h$ is invertible with the inverse $h^{-1}(s) =\frac{s}{1-\abs{z}} - \frac{\abs{z}}{1-\abs{z}}
$ we can conclude that
\[
\int_0^1 \norm{\delta_{h(s)} }_{X^*} ds = \int_{\abs{z}}^1 \norm{\delta_{s} }_{X^*} dh^{-1}(s) = \frac{1}{1-\abs{z}} \int_{\abs{z}}^1 \norm{\delta_{s} }_{X^*}  ds
\]
holds, where the right-hand side is finite according to assumption. Therefore, the function
\[
 \sup_{w\in\D}\abs{K(z,w) }   \frac{ 2      }{  (1-\abs{z})^2   }  \norm{\delta_{h(s)} }_{X^*}\norm{f}_X
\]
is integrable and dominates $ \abs{f(t(s))K(z,t(s)) t'(s)}$, which concludes the proof of
\[
I_K (f)(z) = \int_0^1 T_t(z) f(\phi_t(z)) dt.
\]

Notice that if $K$ is bounded, then $I_K$ is compact, therefore we will assume $K$ is not bounded, or more precisely, that $\lim_{(z,w)\to (1,1)}\abs{K(z,w)} = \infty$.\\

Our main example is  the Hilbert matrix operator
\[
\HH(f)(z) =  \int_0^1 f(x) \frac{1}{1-zx}dx = \int_0^1 w_t(z) f(\phi_t(z)) dt,
\]
which is known to belong to $\mathcal L (A^p_\alpha)$ for $p>2+\alpha>1,$ and to $\mathcal L(H^\infty_\alpha)$ for $\alpha\in(0,1)$.

\bigskip

We finish this section with some more notations. The number $z$ will always be complex in contrast to $r$ which will always be real and non-negative. It usually holds that $\abs{z}<1$ and $0\leq r<1$. If $a\in\closed{\D}$ and $z\to a$, it means that $z\to a$ from inside the open unit disk $\D$. 
We write $f(x)\lesssim g(x)$ if there exists a constant $0<C<\infty$ such that $f(x)\leq Cg(x)$ for all $x$ in some implicit or explicit given set. The notation $\gtrsim $ is used similarly but with $\geq$ instead of $\leq$  and if both hold, we write $f(x)\asymp g(x)$. The Beta function will be denoted by $\beta:\mathbb C\setminus \mathbb Z_{\leq 0} \times \mathbb C\setminus \mathbb Z_{\leq 0} \to \mathbb C$, which should not be mixed up with an open complex ball, $B(w,r)= \{z\in\mathbb C: \abs{z-w}<r\}, \ w\in\mathbb C,$ and $r> 0$. Moreover, the closed unit ball of a Banach space $X$ is denoted by $B_X=\left\{X: \norm{f}_X \leq 1\right\}$.

\section{Some useful results}\label{sec:SomeUsefulResults}

The following lemma is a special case of \cite[Theorem 1]{BL}.
\begin{lem}\label{lem:AnLpIdentity}
Suppose $(f_n)_n\subset A^p_v, \ p\geq 1$, $\sup_n\norm{f_n}_{A^p_v}<\infty$, and $f_n\to f$ pointwise a.e. in $\D$. Then
\[
\lim_{n\to \infty}   (\norm{f_n}^p_{A^p_v} - \norm{f_n-f}^p_{A^p_v}) = \norm{f}^p_{A^p_v}.
\]
\end{lem}

\begin{lem}\label{lem:AResultForALinearCombinationOfDilations}
For the spaces $X=A^p_v, \, 1 < p < \infty$, $X=\BOAFS,$ and $X=\BOAF$, there exists a sequence consisting of compact operators $\LNOP{n}:X\to X, \ n=1,2,\ldots $ such that 
$\limsup_{n\to\infty}\norm{I-\LNOP{n}}_{X\to X}\leq 1$. Moreover, for every $0<R<1$ we have 
\begin{equation}\label{CoreResultForALinearCombinationOfDilations}
\lim_{n\to \infty}\sup_{\norm{f}_X \leq 1} \sup_{\abs{z} \leq R} \abs{(I-\LNOP{n})(f)(z)} = 0.
\end{equation}

\end{lem}

\begin{proof}
The results for $\BOAF$  and  $\BOAFS$  are given explicitly in \cite[Proposition 2.1]{MR}. For the weighted Bergman spaces we invoke Corollary 3.6 in \cite{KW}. Since the space is reflexive, it does not  contain a copy of $\ell^1$. The $(m_p)$ property follows from Lemma \ref{lem:AnLpIdentity}. Using the dilation operator 
$f(z)\mapsto f(rz),  \ 0 < r < 1,$ and the fact that the dilation operator converges strongly to the identity operator $I$  in $A^p_v$, it is easy to see that $A^p_v$ has the metric compact approximation property. It now follows from  \cite[Corollary 3.6 (and the Introduction)]{KW}  that the space of all compact operators $K(A^p_v)$ is an M-ideal in $\mathcal L(A^p_v)$. Using \cite[Theorem 2.4 (4)]{K1},  we obtain a sequence $(\LNOP{n})$ satisfying
\[
\limsup_{n\to\infty}\norm{I-\LNOP{n}}_{A^p_v}\leq 1 \quad  \text{ and } \quad \limsup_{n\to\infty}\norm{(I-\LNOP{n})(f)}_{A^p_v} = 0, \ f\in A^p_v.
\] 
Let $0<R<1$. Since $B_{A^p_v}$ is  compact with respect to the topology of compact convergence, we can find for each $\epsilon>0$ a finite set $\{f_1,\ldots,f_N\}\subset A^p_v$ such that 
\[
\sup_{\norm{f}_X \leq 1} \sup_{\abs{z} \leq R} \abs{(I-\LNOP{n})(f)(z)} \leq \epsilon + \max\left\{ \sup_{\abs{z} \leq R}  \abs{(I-\LNOP{n})(f_j)(z)} : \ j=1,\ldots,N\right\}.
\]
Let $n\to \infty$ to obtain \eqref{CoreResultForALinearCombinationOfDilations}.
\end{proof}

Hereafter, we denote by $\LNOP{n}$ the compact operators given in Lemma \ref{lem:AResultForALinearCombinationOfDilations}.

\section{An upper bound of the essential norm of $I_K$}\label{sec:UpperBound}
In this section we will use the following notations:
\[
 D_{\leq R,t} :=\phi_t(\D)\cap \closed{R\D}  \text{ and } D_{>R,t}  := \phi_t(\D) \setminus \closed{R\D} , \  t,R\in(0,1) .
\]

All results obtained in this  section concerning $I_K:X\to X, \ X=A^p_v$ or $X=\BOAF$ will demand that there exists an $0 < R_0 < 1$ with the following properties. Regarding an operator $I_K$ on $A^p_v$ we have (Condition for Upper Bound $A^p_v$) 
\begin{equation}\tag{CUBA}\label{eq:EssentialNormBddA}
  \int_0^1       \sup_{ z\in \phi_t^{-1}(D_{>R_0,t}) }  \frac{  \abs{   T_t(z)   }   v(z)^\frac{1}{p}  }{ v(\phi_t(z))^\frac{1}{p} } \frac{t^\frac{2}{p}   dt   }{(1-t)^\frac{2}{p} }<\infty
\end{equation}
and for every $R<1$ it holds that
\begin{equation}\tag{CUB2A}\label{eq:ConditionForUpperBoundOfIKBergman}
\int_0^1     \norm{T_t\chi_{   \phi_t^{-1}(D_{\leq R,t})   }  }_{L^p_v}   dt < \infty.
\end{equation}
If we instead have $X=\BOAF$, we will demand
\begin{equation}\tag{CUBH}\label{eq:EssentialNormBddH}
  \int_0^1       \sup_{ z\in \phi_t^{-1}(D_{>R_0,t}) }  \frac{  \abs{   T_t(z)   }   \BOAFW(z)  }{ \BOAFW(\phi_t(z)) }dt <\infty 
\end{equation}
and
\begin{equation}\tag{CUB2H}\label{eq:ConditionForUpperBoundOfIKInfty}
\int_0^1     \norm{T_t\chi_{   \phi_t^{-1}(D_{\leq R,t})   }  }_{L^\infty_\BOAFW}   dt < \infty
\end{equation}
for every $R<1$.

\begin{rem}\label{rem:ProblemWithUnboundedWeights}

Notice that if $v$ is unbounded, then 
\[
\inf_{t\in(0,1)}\limsup_{z\to 1}\frac{v(z)}{v(\phi_t(z))}=\infty.
\]
Indeed, since $v$ is radial we may write the quotient as $\frac{v(\abs{z})}{v(\abs{\phi_t(z)})}$. Let $0<t<1$ and notice that $\abs{\phi_t(re^{i\theta})}<M(\theta)<1, \ \theta \neq 0$. Hence, we have $\frac{v(r)}{v(\phi_t(re^{i\theta}))}\to \infty$ as $r\to 1$, when $ \theta \in (-\pi,\pi)\setminus\{0\}$. Choose a strictly decreasing positive sequence $(\theta_n)_n$ such that $\lim_{n\to\infty}\theta_n\to 0$ and let $a(n)$ be an increasing function, tending to $\infty$ as $n\to\infty$. For each $n$, define $r_n$ such that $\frac{v(r_n)}{v(\phi_t(r_ne^{i\theta_n}))}>a(n)$ and the statement follows using the sequence $(r_ne^{i\theta_n})_n$. Notice that $a(n)\to \infty$ as $n\to\infty$ arbitrarily fast. For this reason one can expect that condition (\ref{eq:EssentialNormBddA}) is in general not satisfied if the weight is unbounded.  
\end{rem}

\begin{thm}\label{thm:Upperbound}

Let $1 <  p < \infty.$ For operators $I_K:A^p_v \to A^p_v$, we have

\begin{equation}\label{eq:ConditionForUpperBoundOfIKBergman2}
 \norm{I_K}_{e,A^p_v\to A^p_v} \leq  \int_0^1       \limsup_{ z\to 1  }  \frac{  \abs{   T_t(z)   }   v(z)^\frac{1}{p}  }{ v(\phi_t(z))^\frac{1}{p} } \frac{t^\frac{2}{p}   dt   }{(1-t)^\frac{2}{p} }
\end{equation}
and for operators $I_K:\BOAF \to \BOAF$ (or $\BOAFS\to \BOAF$), 
\begin{equation}\label{eq:ConditionForUpperBoundOfIK}
 \norm{I_K}_{e,\BOAF\to \BOAF} \leq  \int_0^1       \limsup_{ z\to 1  }  \frac{  \abs{   T_t(z)   }   \BOAFW(z) }{ \BOAFW(\phi_t(z)) } dt \quad \text{ and }
\end{equation}
\[
 \norm{I_K}_{e,\BOAFS\to \BOAF} \leq  \int_0^1       \limsup_{ z\to 1  }  \frac{  \abs{   T_t(z)   }   \BOAFW(z) }{ \BOAFW(\phi_t(z)) } dt.
\]
\end{thm}
\begin{proof}  We discuss each of the cases  $A^p_v$ and    $\BOAF$, $\BOAFS$  separately.

We start with $A^p_v$. For any $f\in  A^p_v$,
we have
\[
I_K(f)(z) = \int_0^1  f(\phi_t(z))  T_t(z) dt.
\]
By Minkowski's inequality for integrals, we obtain
\begin{align*}
\norm{I_K(f)}_{A^p_v}& \leq \int_0^1 \norm{T_t C_{\phi_t}(f)}_{A^p_v} dt \\
&= \int_0^1  \left(    \int_\D       \abs{  f(\phi_t(z))  }^p \abs{   T_t(z)  }^p  v(z)  \frac{dA(z)}{\pi}    \right)^\frac{1}{p} dt \\
 &= \int_0^1  \left(    \int_{\phi_t(\D )}      \abs{  f(w)  }^p \abs{   T_t(\phi_t^{-1}(w))   }^p  v(\phi_t^{-1}(w))  \abs{{(\phi_t^{-1})}'(w)}^2  \frac{dA(w)}{\pi}    \right)^\frac{1}{p} dt\\
 &= \int_0^1   \frac{t^\frac{2}{p}}{(1-t)^\frac{2}{p}}     \left(   \int_{\phi_t(\D )}  \abs{w}^{-4}   \abs{  f(w)  }^p  \abs{   T_t(\phi_t^{-1}(w))   }^p v(\phi_t^{-1}(w))   \frac{dA(w)}{\pi}    \right)^\frac{1}{p} dt.
\end{align*}
Next, we split the integral over $\phi_t(\D )$ in two parts
\begin{align*}
\frac{(1-t)^2} {t^2}  \norm{T_t C_{\phi_t}(f)}^p_{A^p_v} &=   \int_{        D_{\leq R,t}     }  \abs{w}^{-4}   \abs{  f(w)  }^p  \abs{   T_t(\phi_t^{-1}(w))   }^p v(\phi_t^{-1}(w))   \frac{dA(w)}{\pi}  \\
&\quad  +           \int_{      D_{> R,t}     }  \abs{w}^{-4}   \abs{  f(w)  }^p  \abs{   T_t(\phi_t^{-1}(w))   }^p v(\phi_t^{-1}(w))   \frac{dA(w)}{\pi}  \\
&  \leq  \sup_{\abs{z}\leq R}  \abs{f(z)}^p    \int_{        D_{\leq R,t}     }  \abs{w}^{-4}     \abs{   T_t(\phi_t^{-1}(w))   }^p v(\phi_t^{-1}(w))   \frac{dA(w)}{\pi}  \\
&\quad  +   \sup_{z\in   D_{> R,t}} \left(    \abs{z}^{-4} \abs{   T_t(\phi_t^{-1}(z))   }^p   \frac{ v(\phi_t^{-1}(z))  }{ v(z) }  \right)    \int_{      D_{> R,t}     }   \abs{  f(w)  }^p v(w) \frac{dA(w)}{\pi}.  \\
\end{align*}
Since $(a+b)^\frac{1}{p}\leq a^\frac{1}{p} + b^\frac{1}{p} , \ a,b\geq 0,$ we conclude 
\begin{align*}
 \norm{T_t C_{\phi_t}(f)}_{A^p_v} &\leq  \sup_{\abs{z}\leq R}  \abs{f(z)}  \norm{T_t\chi_{   \phi_t^{-1}(D_{\leq R,t})   }  }_{L^p_v}  \\
&\quad  +   \frac{t^{\frac{2}{p}}}{(1-t)^\frac{2}{p} }   \sup_{z\in   D_{> R,t}} \left(    \abs{z}^{-4} \abs{   T_t(\phi_t^{-1}(z))   }^p   \frac{ v(\phi_t^{-1}(z))  }{ v(z) }  \right)^\frac{1}{p}    \norm{f}_{A^p_v} \\
 &\leq  \sup_{\abs{z}\leq R}  \abs{f(z)} \norm{T_t\chi_{   \phi_t^{-1}(D_{\leq R,t})   }  }_{L^p_v}   + \norm{f}_{A^p_v} \frac{   t^{\frac{2}{p}}    R^{-\frac{4}{p}}      }{(1-t)^\frac{2}{p} }   \sup_{z\in  \phi_t^{-1}(  D_{> R,t}  )}   \frac{  \abs{   T_t(z)   }   v(z)^\frac{1}{p}  }{ v(\phi_t(z))^\frac{1}{p} }.
\end{align*}

 Applying the calculations above to $f-\LNOP{n}(f)$ and taking the supremum over $f\in B_{A^p_v}$, we obtain
\begin{align*}
\norm{I_K-I_K\LNOP{n}}_{A^p_v\to  A^p_v} &\leq \sup_{f\in B_{A^p_v}} \sup_{\abs{z}\leq R}  \abs{(I-\LNOP{n})(f)(z)} \int_0^1     \norm{T_t\chi_{   \phi_t^{-1}(D_{\leq R,t})   }  }_{L^p_v}   dt  \\
&\quad + R^{-\frac{4}{p}}    \norm{ I-\LNOP{n} }_{A^p_v\to A^p_v}    \int_0^1 \frac{   t^{\frac{2}{p}}     }{(1-t)^\frac{2}{p} }   \sup_{z\in  \phi_t^{-1}(  D_{> R,t}  )}   \frac{  \abs{   T_t(z)   }   v(z)^\frac{1}{p}  }{ v(\phi_t(z))^\frac{1}{p} }  dt.
\end{align*}
Using Lemma \ref{lem:AResultForALinearCombinationOfDilations} and (\ref{eq:ConditionForUpperBoundOfIKBergman}) we get, for every $0<R<1$
\[
\limsup_{n\to\infty}\norm{I_K- I_K\LNOP{n}}_{A^p_v\to  A^p_v} \leq R^{-\frac{4}{p}}     \int_0^1 \frac{   t^{\frac{2}{p}}     }{(1-t)^\frac{2}{p} }   \sup_{z\in  \phi_t^{-1}(  D_{> R,t}  )}   \frac{  \abs{   T_t(z)   }   v(z)^\frac{1}{p}  }{ v(\phi_t(z))^\frac{1}{p} }  dt.
\]
Letting $R\to 1$ and using the dominated convergence theorem we have
\[
 \norm{I_K}_{e,A^p_v\to A^p_v} \leq  \int_0^1       \lim_{R\to 1  }  \kern  -2pt \sup_{z\in  \phi_t^{-1}(  D_{> R,t}  )} \kern - 5pt  \frac{  \abs{   T_t(z)   }   v(z)^\frac{1}{p}  }{ v(\phi_t(z))^\frac{1}{p} } \frac{t^\frac{2}{p}   dt   }{(1-t)^\frac{2}{p} }      =\int_0^1       \limsup_{ z\to 1  }  \frac{  \abs{   T_t(z)   }   v(z)^\frac{1}{p}  }{ v(\phi_t(z))^\frac{1}{p} } \frac{t^\frac{2}{p}   dt   }{(1-t)^\frac{2}{p} } ,
\]
where we have used the fact that $\phi_t$ is injective with a continuous inverse and $1$ as a fixed point, to obtain the last equality. The use of the dominated convergence theorem is justified by (\ref{eq:EssentialNormBddA}).

Next, we consider the spaces  $\BOAF$ and $\BOAFS$.
This case could be proved  similarly  to the weighted Bergman case. We might, however, use the calculations done for the $A^p_v$-case. Notice that for $f\in \BOAF$ we have
\[
\norm{f}_{\BOAF} = \lim_{p\to\infty}\norm{f}_{A^p_{\BOAFW^p}},
\]
where the norms on the right-hand side can be viewed as an increasing sequence. Hence, we immediately obtain for $f\in \BOAF$ or $f\in \BOAFS$
\begin{align*}
\norm{ T_t C_{\phi_t}(f) }_{\BOAF} &\leq \lim_{p\to\infty}  \left( \sup_{\abs{z}\leq R}  \abs{f(z)} \norm{T_t\chi_{   \phi_t^{-1}(D_{\leq R,t})   }  }_{L^p_{\BOAFW^p}}   + \norm{f}_{A^p_{\BOAFW^p}} \frac{   t^{\frac{2}{p}}    R^{-\frac{4}{p}}      }{(1-t)^\frac{2}{p} }   \sup_{z\in  \phi_t^{-1}(  D_{> R,t}  )}   \frac{  \abs{   T_t(z)   }   \BOAFW(z)  }{ \BOAFW(\phi_t(z)) }  \right)  \\
 &=   \sup_{\abs{z}\leq R}  \abs{f(z)} \norm{T_t\chi_{   \phi_t^{-1}(D_{\leq R,t})   }  }_{L^\infty_\BOAFW}   + \norm{f}_{\BOAF}   \sup_{z\in  \phi_t^{-1}(  D_{> R,t}  )}   \frac{  \abs{   T_t(z)   }   \BOAFW(z) }{ \BOAFW(\phi_t(z)) } .
\end{align*}
The rest of the proof is  similar  to the $A^p_v$-case.

\end{proof}

\section{A lower bound of the essential norm of $I_K$ on $A^p_v$}\label{sec:LowerBoundOfIKOnAp}
Most of the results obtained in this section do not demand that $v$ is bounded, however, condition (\ref{eq:StrongCorrelationBetween_v_and_gA}), which is presented later, does not hold if $v$ is unbounded, since the limit is not finite in this case according to Remark \ref{rem:ProblemWithUnboundedWeights}.

We begin by introducing some more definitions and assumptions. Let
\[
f_c(z) := \frac{g(z)}{(1-z)^c},\ c\in(0,\frac{2}{p}].
\]
Most of the time we will, however, use $c<\frac{2}{p}$. The reason for this is clear considering the next condition. We will assume that there exists a constant $\Omega\in(\frac{1}{p},\frac{2}{p})$ such that (Condition Approximate Identity $A^p_v$)
\begin{equation}\tag{CAIA}\label{eq:ConditionApproximateidentityA}
f_c\in A^p_v \text{ when } \Omega<c<\frac{2}{p}\text{ and }\lim_{c\to\frac{2}{p}}\norm{f_c}_{A^p_v}=\infty. 
\end{equation}
This condition is always assumed to hold. If we need any of the conditions presented below, we will write them explicitly in the statements of our results, lemmas and remarks. In this section the lower bound of the essential norm will be based on the following functions
\begin{align*}
C_c(z)&:=  \int_0^1 (\frac{1}{1-t}-z)^c\frac{g(\phi_t(z))}{g(z)}  T_t(z) dt, \ 0<c\leq \frac{2}{p},
\end{align*}
which are meromorphic in $\D$.
We will assume that there exist constants $\zeta>0$ and $\frac{1}{p}<\Omega<\frac{2}{p}$ such that (Condition concerning the Kernel function K)
\begin{equation}\tag{CKA}\label{eq:EquivToIKBounded}
\int_0^1  \sup_{c\in(\Omega,\frac{2}{p})}   \sup_{z\in B(1,\zeta) \cap \D}     \abs{   (\frac{1}{1-t}-z)^c\frac{g(\phi_t(z))}{g(z)}   T_t(z)  }dt <\infty.
\end{equation}

To obtain the exact value of the essential norm, we also assume that for every $t\in(0,1)$ we have
\begin{equation}\tag{CEVA}\label{eq:StrongCorrelationBetween_v_and_gA}
\limsup_{z\to 1} \frac{v(z)}{v(\phi_t(z))} = \lim_{z \to 1}    \frac{g(\phi_t(z))^p}{g(z)^p}<\infty \text{ and } \lim_{z\to 1} T_t(z)\in [0,\infty)
\end{equation}
(Condition for Exact Value).

To increase the generality of some results, we introduce a weaker version of (\ref{eq:EquivToIKBounded}), namely, 
\begin{equation}\tag{CKA-}\label{eq:EquivToIKBoundedWeak}
C_\infty := \sup_{c\in(\Omega,\frac{2}{p})}   \sup_{z\in B(1,\zeta) \cap \D}   \abs{C_c(z)}  =  \sup_{c\in(\Omega,\frac{2}{p})}    \sup_{z\in B(1,\zeta) \cap \D}    \abs{ \int_0^1       (\frac{1}{1-t}-z)^c\frac{g(\phi_t(z))}{g(z)}  T_t(z)  dt } <\infty
\end{equation}
for some $\zeta>0$ and $\frac{1}{p}<\Omega<\frac{2}{p}$ and a weaker version of 
(\ref{eq:StrongCorrelationBetween_v_and_gA}), which states that the limit
\begin{equation}\tag{CEVA-}\label{eq:ConditionForSimilarExpressionA}
\lim_{z \to 1} \frac{g(\phi_t(z))}{g(z)} T_t(z)
\end{equation}
exists. The notations $\Omega,\zeta$ and $C_\infty$ will be used for this purpose in this section and Section \ref{sec:Examples}.

\begin{lem}\label{lem:IKoperatingOnFcIsMultiOp}
For $c\in (0,\frac{2}{p}]$ and $z\in\D\setminus P_c$ we have
\[
I_K (f_c)(z) =  f_c(z) C_c(z),
\]
where $P_c$ consists of the poles of $C_c$.
\end{lem}
\begin{proof}
We have
\begin{align*}
I_K (f_c)(z) &= \int_0^1 T_t(z) f_c(\phi_t(z)) dt  = \frac{1}{(1-z)^c} \int_0^1 T_t(z) \left(\frac{1-(1-t)z}{1-t} \right)^c  g(\phi_t(z))  dt  =  f_c(z) C_c(z).
\end{align*}
\end{proof}

It is clear that for $h_c:=\frac{f_c}{\norm{f_c}_{A^p_v}  }$ we have for all $z\in\D$, $h_c(z)\to 0 $ as $c\to \frac{2}{p}$. Since the space $A^p_v$ is reflexive, the closed  unit ball $B_{A^p_v}$ is weakly compact. Therefore by a standard argument, we obtain that $h_c\to 0$ weakly, as $c\to\frac{2}{p}$. Next, we present the two main results of this section (Theorem \ref{thm:bestlowerbdd} and Corollary  \ref{cor:EssentalNormOnA}).

\begin{thm}
\label{thm:bestlowerbdd}
Assume that condition (\ref{eq:EquivToIKBoundedWeak}) holds. Let 
\[
C=  \lim_{(c,z)\to(\frac{2}{p},1)} \inf_{\gamma\in(c,\frac{2}{p})} \inf_{w\in  (B(1,\abs{1-z})\cap \D)}  \abs{C_\gamma(w)}. 
\]
We have
\begin{equation}\label{eq:ObtainingLowerBound}
\liminf_{c\to \frac{2}{p}}  \norm{   I_K\Big(\frac{f_c}{\norm{f_c})_{A^p_v}} \Big)  }_{A^p_v} \geq C
\end{equation}
and $C$ is a lower bound of $\norm{I_K}_{e,A^p_v\to A^p_v}$. 
\end{thm}
\begin{proof}
Given (\ref{eq:ObtainingLowerBound}) it follows from $h_c=\frac{f_c}{\norm{f_c}_{A^p_v}}$ being a weak null sequence that $\norm{L(h_c)}_{A^p_v}\to 0 $ as $c\to\frac{2}{p}$ for every compact operator $L\in \mathcal L (A^p_v)$. Using the triangle inequality, we obtain $C\leq \norm{I_K}_{e,A^p_v\to A^p_v}$.

Define

\[
S_c(z) = \left\{
\begin{array}{@{}l l}
 \inf_{\gamma\in(c,\frac{2}{p})} \inf_{w\in  B(1,\abs{1-z})\cap \D}  \abs{C_\gamma(w)},&  z\in B(1,\zeta) \cap \D\\
0, & z\in \D \setminus B(1,\zeta).
\end{array}\right.
\]
It is clear that the limit 
\[
C =  \lim_{(c,z)\to (\frac{2}{p},1)}  S_c(z)
\]
exists. Now, choose $0<\rho<\zeta$ and $\Omega<c_1<\frac{2}{p}$ such that for $z\in M=B(1,\rho)\cap \D$ and $c_1<c<\frac{2}{p}$ we have $\abs{S_c(z)-C}<\epsilon$. Finally, choose $c_1<c_2<\frac{2}{p}$ such that $\sup_{ z\in  \D\setminus M   }\abs{h_c(z)}<\epsilon$ whenever $c_2<c<\frac{2}{p}$. For these parameters we have, using $M' = \D\setminus M$, that
\begin{align*}
\abs{    \norm{    h_c  S_c    }_{A^p_v} -C   }^p   &\leq \int_M \abs{    h_c(z)  (S_c(z)  -C)  }^p  dA_v(z)  + \int_{ M'}   \abs{    h_c(z)  (S_c(z)  -C)  }^p  dA_v(z) \\
&<\epsilon^p  +   \epsilon^p \int_{M'}   \abs{     (S_c(z)  -C)  }^p  dA_v(z)  \\
&< \epsilon^p  +   2^p\epsilon^p   C_\infty^p.
\end{align*}

\begin{figure}[h]

  \centering

\begin{tikzpicture}[scale=2]
	\draw (0,0) circle (1);

\draw (0.707,0.707) arc (112.5:247.5:0.765);

\draw[dashed](1,0) -- (0.459,0.541);
		\node[left] at (0.65,0.3) {$\rho$};
	\node[left] at (0.7,0) {$M$};
\node[left] at (0,0) {$M'$};
	


\end{tikzpicture} 
   \caption{A partition of $\D$.}
\end{figure}

We have now proved that there is a constant $c_2 <\frac{2}{p}$ such that for $c\in(c_2,\frac{2}{p})$ we have 
\[
  \norm{    h_c  C_c    }_{A^p_v}  \geq  \norm{    h_c  S_c    }_{A^p_v} > C - \epsilon \left(1 + 2^p C_\infty^p \right)^\frac{1}{p}.
\]
Hence, after letting $c\to \frac{2}{p}$ followed by $\epsilon \to 0$ we obtain $ \liminf_{c\to\frac{2}{p}} \norm{    h_c  C_c    }_{A^p_v}  \geq C$. Now Lemma \ref{lem:IKoperatingOnFcIsMultiOp} gives us (\ref{eq:ObtainingLowerBound}).

\end{proof}

\begin{cor}\label{cor:EssentalNormOnA}
Assume that condition (\ref{eq:EquivToIKBounded}) holds. If condition (\ref{eq:ConditionForSimilarExpressionA}) holds, we have that
\[
C \geq \abs{\int_0^1\left(\frac{t}{1-t}\right)^\frac{2}{p}\lim_{z \to 1} \frac{g(\phi_t(z))}{g(z)} T_t(z) dt}.
\]
If the stronger condition (\ref{eq:StrongCorrelationBetween_v_and_gA}) is satisfied, then
\[
\norm{I_K}_{e,A^p_v \to A^p_v} = \int_0^1\left(\frac{t}{1-t}\right)^\frac{2}{p}\limsup_{z\to 1} \frac{v(z)^\frac{1}{p}}{v(\phi_t(z))^\frac{1}{p}} T_t(z) dt.
\]
\end{cor}
\begin{proof}
Let $GT(t) = \lim_{z \to 1} \frac{g(\phi_t(z))}{g(z)}  T_t(z) $ and
\[
J_{t,\gamma}(z):= \left(\frac{1}{1-t}-z\right)^\gamma \frac{g(\phi_t(z))}{g(z)} T_t(z).
\]
First, by using the dominated convergence theorem, which can be applied as done below due to (\ref{eq:EquivToIKBounded}), we have
\begin{align*}
  &\lim_{(c,z)\to(\frac{2}{p},1)} \sup_{\gamma\in(c,\frac{2}{p})} \sup_{w\in  B(1,\abs{1-z})\cap \D}  \int_0^1   \abs{ J_{t,\gamma}(w) - \left(\frac{t}{1-t}\right)^\frac{2}{p}GT(t)  } dt \\
&\leq  \lim_{(c,z)\to(\frac{2}{p},1)}  \int_0^1  \sup_{\gamma\in(c,\frac{2}{p})} \sup_{w\in  B(1,\abs{1-z})\cap \D}  \abs{ J_{t,\gamma}(w) - \left(\frac{t}{1-t}\right)^\frac{2}{p}GT(t)  } dt \\
&=  \int_0^1 \lim_{(c,z)\to(\frac{2}{p},1)}    \sup_{\gamma\in(c,\frac{2}{p})} \sup_{w\in  B(1,\abs{1-z})\cap \D}  \abs{ J_{t,\gamma}(w) - \left(\frac{t}{1-t}\right)^\frac{2}{p}GT(t)  } dt \\
&= 0.
\end{align*}

Since 
\begin{align*}
\abs{C_\gamma(w)} &\geq  \abs{   \int_0^1\left(\frac{t}{1-t}\right)^\frac{2}{p}GT(t) dt    } - \abs{ \int_0^1 J_{t,\gamma}(w)  - \left(\frac{t}{1-t}\right)^\frac{2}{p}GT(t)  dt}\\
&\geq  \abs{   \int_0^1\left(\frac{t}{1-t}\right)^\frac{2}{p}GT(t) dt    } -  \int_0^1  \abs{  J_{t,\gamma}(w)  - \left(\frac{t}{1-t}\right)^\frac{2}{p}GT(t) }  dt\\
\end{align*}
and the first statement in the corollary follows. The second statement follows directly from the first and condition (\ref{eq:StrongCorrelationBetween_v_and_gA}). Notice that condition (\ref{eq:StrongCorrelationBetween_v_and_gA}) implies that $ \lim_{z \to 1} \frac{g(\phi_t(z))}{g(z)}  \in[0,\infty)$.  

\end{proof}

We will finish this section with some practical results.

\begin{thm}\label{thm:SomeSuffCond}
If $v$ is bounded, then a sufficient condition for (\ref{eq:ConditionApproximateidentityA}) is that
\begin{enumerate}[itemsep=5pt]
\item 
$\sup_{z\in\D}\abs{g(z)}^p v(z)<\infty;$ 

\item
$\abs{g(z)}^p v(z) \text{ is bounded away from zero on a non-zero Stolz sector close enough to } 1.$
\end{enumerate}

\end{thm}
\begin{proof}
Let $\tilde{\Delta}_{d,l}, \ 0<l<d<1,$ be the triangle with vertices at points $1-d+li, 1-d-li$ and $1$. By assumption, there exists constants $0<l_0<d_0<1$ such that $\abs{g(z)}^p v(z) $ is bounded away from zero on $\tilde{\Delta} := \tilde{\Delta}_{d_0,l_0}$ by a constant $R_0>0$. Hence, we have
\begin{align*}
\norm{f_c}_{A^p_v}^p &\geq  \int_{\tilde{\Delta}} \frac{R_0  dA(z)        }{\pi \abs{1-z}^{c p}   } 
=\int_{1-\tilde{\Delta}} \frac{R_0  dA(z)        }{\pi \abs{z}^{c p}   }
=2\int_0^{d_0} \int_0^{\frac{l_0}{d_0}x}    \frac{R_0  dy  dx      }{\pi  (x^2+y^2)^{\frac{c p}{2}}   }\\
&=2\int_0^{d_0} \int_0^\frac{l_0}{d_0}    \frac{R_0  xdy dx       }{ x^{c p}\pi  (1+y^2)^{\frac{c p}{2}}   }     
  \geq 2R_0  \int_0^{d_0} x^{1-c p }dx  \int_0^\frac{l_0}{d_0}    \frac{  dy       }{ \pi (1+y^2)   }       
 \asymp \frac{d_0^{2-cp}}{ 2-cp},
\end{align*}
which proves that $\lim_{c\to \frac{2}{p}}\norm{f_c}_{A^p_v}=\infty$. Furthermore, using $R_\infty=\sup_{z\in\D}\abs{g(z)}^p v(z)<\infty$, we have for fixed $c\in(\frac{1}{p},\frac{2}{p})$ that
\begin{align*}
 \int_{\D}  \abs{f_ c(z)   }^p dA_v & \leq   R_\infty \int_{\D}  \frac{dA(z)}{    \pi \abs{1-z}^{c p}      } \lesssim   \int_0^1(1-r)^{1-cp}dr =\left[\frac{-(1-r)^{2-cp}}{2-cp}\right]_0^1 <\infty.
\end{align*}
The constant involved in $\lesssim$ depends on $cp$ (see Lemma \ref{lem:ProvingCAIA} for details).

\end{proof}

\begin{rem}\label{rem:usefulRemark}
Considering the case when $v(z) = g(\abs{z})^{-p}$ is real, $v$ is non-increasing and $\sup_{\theta\in(-\pi,\pi)} \abs{T_t(re^{i\theta})} \leq M\abs{T_t(r)}$ for some constant $M$, we can obtain a useful condition that is stronger than both (\ref{eq:EssentialNormBddA}) and (\ref{eq:EquivToIKBounded}). Notice that 
\begin{align*}
\left(\frac{v(z)}{v(\phi_t(z))} \right)^\frac{1}{p}  = \frac{g(\abs{\phi_t(z)})}{g(\abs{z})} \leq \frac{g(\phi_t(\abs{z}))}{g(\abs{z})} . 
\end{align*}
from which it follows that
\begin{align*}
 \sup_{ z\in \D }  \frac{  \abs{   T_t(z)   }   v(z)^\frac{1}{p}  }{ v(\phi_t(z))^\frac{1}{p} }  \frac{t^\frac{2}{p}   dt   }{(1-t)^\frac{2}{p} } &\leq  \sup_{r\in(0,1)} T_t(r) \frac{t^\frac{2}{p}  }{(1-t)^\frac{2}{p} }      \sup_{ \rho\in (0,1) } \frac{  g(\rho) }{ g(\phi_t(\rho)) } \\
&\leq  \sup_{c\in(\Omega,\frac{2}{p})} \sup_{r\in(0,1)} T_t(r) \left(   \frac{1  }{ 1-t } - r  \right)^c      \sup_{ \rho\in (0,1) } \frac{  g(\rho) }{ g(\phi_t(\rho)) }.
\end{align*}
It is now easy to see that, if there exists $\frac{1}{p}<\Omega<\frac{2}{p}$ such that
\begin{equation}\label{eq:StrongerThanCubaNCka}
\int_0^1  \sup_{c\in(\Omega,\frac{2}{p})}   \sup_{z\in  \D}   \abs{  (\frac{1}{1-t}-z)^c     T_t(z)       } \sup_{z\in \D}   \abs{ \frac{g(\phi_t(z))}{g(z)} }dt <\infty,
\end{equation}
then condition (\ref{eq:EssentialNormBddA}) holds. From (\ref{eq:StrongerThanCubaNCka}) it trivially follows that condition
(\ref{eq:EquivToIKBounded}) is satisfied without the three extra assumptions given in the beginning of this remark.

\end{rem}

\section{A lower bound of the essential norm of $I_K$ on $\BOAF$}\label{sec:LowerBoundOfIKOnHInfty}


Consider the functions
\[
f_{c,n}(z) := z^n  g(z)(1-z)^c , \ c\geq 0,n\in\mathbb Z_{\geq 1},
\]
where $g\in \HOLO(\D)$ is assumed satisfy the conditions
\begin{equation}\tag{C1H}\label{eq:WeightCondUpperBoundWeakH}
\sup_{z\in\D}\abs{g(z)} \BOAFW(z) \leq 1
\end{equation}
and
\begin{equation}\tag{C1H+}\label{eq:WeightCondUpperBoundStrong}
\begin{array}{l}
\limsup_{z\to e^{it}} \abs{g(z)} v(z) = 0 \text{ unless } t=0,\\
\abs{g(r)} \BOAFW(r)\to1  \text{ as }  r\to 1.
\end{array}
\end{equation}

To obtain the exact value of the essential norm, we have to assume  that for every $t\in(0,1)$ we have
\begin{equation}\tag{CEVH}\label{eq:StrongCorrelationBetween_v_and_gH}
 \lim_{r\to 1}    \frac{g(\phi_t(r))}{g(r)} = \limsup_{z\to 1} \frac{\BOAFW(z)}{\BOAFW(\phi_t(z))}<\infty \text{ and } \lim_{z\to 1} T_t(z) \in [0,\infty).
\end{equation}

We also obtain a lower bound using a weaker assumption, that for all $r,t\in(0,1)$ 
\begin{equation}\tag{CEVH-}\label{eq:ConditionForSimilarExpressionH}
\frac{   g(\phi_t(r))   }{g(r)}     T_t(r)  \in[0,\infty).
\end{equation}

Define
\[
C_{c,n}(r) :=\int_0^1   (\frac{1}{1-t}-r)^{-c}                 \left(\frac{   \phi_t(r)    }{r }\right)^n \frac{   g(\phi_t(r))   }{g(r)}               T_t(r)   dt, \ \ \ c\geq 0, n\in\mathbb Z_{\geq 1}, r\in(\delta,1)
\]
for some $0<\delta<1$. This function will form the foundation for the value of the essential norm.

 Moreover, since
\[
  \lim_{r\to 1} \lim_{c\to 0} \abs{f_{c,n}(r)} \BOAFW(r) = 1,
\]
we have
\[
\sup_{0<c<\frac{1}{n}}  \sup_{r\in (1-\gamma(n), 1)} \abs{f_{c,n}(r)} \BOAFW(r) \geq 1,
\]
where $0 < \gamma(n) < 1- \delta$ for all $n$. We can now, for every $n$, choose a constant $c_n\in (0,\frac{1}{n} )$ such that
\begin{equation}\label{AnUsefulIneqFor}
 \sup_{r\in (1-\gamma(n), 1)}   \abs{f_{c_n,n}(r)} \BOAFW(r)>1-\frac{1}{n}.
\end{equation}
Let $f_n:= \frac{       f_{c_n,n}      }{      \norm{   f_{c_n,n}   }_{\BOAF}        }$. Notice that for all $n$ we have $f_n\in \BOAFS$ and $f_n \to 0$ uniformly  on compact subsets of $\D$, when $n\to\infty.$

\begin{thm}\label{thm:LowerboundMainH}
Let $(\gamma(n))_n\subset(0,1-\delta)$ be a sequence such that $\lim_{n\to \infty} n\gamma(n)=0$. We have, for $X=\BOAF$ or $X=\BOAFS$,
\begin{align*}
\norm{  I_K  }_{e,X \to \BOAF  } &\geq    
\liminf_{n\to \infty} \inf_{r\in  (1-\gamma(n),1)      } \abs{   C_{c_n,n}(r)   }
\end{align*}
and if condition (\ref{eq:ConditionForSimilarExpressionH}) holds, we obtain
\[
\norm{  I_K}_{e,X \to \BOAF  } \geq    \int_0^1         \liminf_{r\to 1}  \frac{   g(\phi_t(r))   }{g(r)}             T_t(r)  dt    .         
\]
\end{thm}
\begin{proof}

We have    
\[
\inf_{1-\gamma(n)<r<1} \left( \frac{\phi_t(r)}{r}\right)^n \geq \left( \phi_t(1-\gamma(n))\right)^n \to 1
\]
as $n\to\infty$. Also,
\[
\inf_{1-\gamma(n)<r<1} \left( \frac{\phi_t(r)}{r}\right)^n  \leq \frac{1}{  (1-\gamma(n))^n  } \to 1
\]
as $n\to\infty$, and hence, the second statement follows from the first by taking the infimum inside the integral followed by Fatou's lemma.
For the first statement we have
\begin{align*}
I_K(f_n)(r) &= \int_0^1 f_n(\phi_t(r)) T_t(r) dt \\ 
& = \frac{   (1-r)^{c_n}  }{      \norm{   f_{c_n,n}   }_{\BOAF}        }  \int_0^1 \phi_t(r)^n   g(\phi_t(r))  \frac{(1-t)^{c_n}   }{((t-1)r+1) ^{c_n}  }   T_t(r)   dt\\
& = f_n(r)  \frac{1}{r^n g(r)} \int_0^1 \phi_t(r)^n g(\phi_t(r))   (\frac{1}{1-t}-r)^{-c_n}  T_t(r)   dt\\
&=f_n(r) C_{{c_n},n}(r).
\end{align*}
Observe that the first equality does not require $\int_0^1 \norm{\delta_s}_{X^*} ds <\infty$, because the integration path $[0,1)$ is mapped onto $[0,1)$. Now, it follows from condition (\ref{eq:WeightCondUpperBoundWeakH}) and (\ref{AnUsefulIneqFor}) that

\begin{align*}
\norm{I_K f_n}_{\BOAF} &\geq  \sup_{r\in (1-\gamma(n),1)}   \abs{f_n(r) C_{{c_n},n}(r)}\BOAFW(r) \geq  \sup_{r\in (1-\gamma(n),1)}    \abs{f_n(r)} \BOAFW(r)  \inf_{r'\in  (1-\gamma(n),1)                   }   \abs{  C_{{c_n},n}(r')} \\ 
&\geq    \frac{1}{          \norm{   f_{c_n,n}   }_{\BOAF}       }\left(1- \frac{1}{n} \right)    \inf_{r'\in  (1-\gamma(n),1)                   }   \abs{  C_{{c_n},n}(r')} \geq    \frac{1}{       2^{c_n}     }\left(1- \frac{1}{n} \right)    \inf_{r'\in  (1-\gamma(n),1)                   }   \abs{  C_{{c_n},n}(r')} \\ 
\end{align*}
so that
\begin{align*}
\liminf_{n\to \infty}\norm{I_K f_n}_{\BOAF} & \geq \liminf_{n\to \infty} \inf_{r'\in  (1-\gamma(n),1)                   }   \abs{  C_{{c_n},n}(r')}.
\end{align*}
Since $(f_n)_n\subset \BOAFS$ is a bounded sequence that  converges to zero uniformly on compact subsets of $\D$, a standard argument yields that  
$f_n\to 0$ weakly in $\BOAFS$ (see \cite{MR}).  Therefore, for every compact operator $L\in \mathcal L(\BOAFS)$ we have  $\norm{L f_n}_{\BOAF}\to 0$ as $n\to \infty$, which gives the statement for $X=\BOAFS$. Moreover, continuity of the inclusion map $\BOAFS \xhookrightarrow{} \BOAF$ implies that $f_n\to 0$ weakly in $\BOAF$ and this completes the proof.

\end{proof}

\begin{cor}\label{cor:EssentalNormOnH}
If condition (\ref{eq:StrongCorrelationBetween_v_and_gH}) holds, then we have for $X=\BOAF$ or $X=\BOAFS$,
\[
\norm{  I_K  }_{e,X \to \BOAF  } = \int_0^1 \limsup_{z\to 1} \frac{\BOAFW(z)}{\BOAFW(\phi_t(z))} T_t(z) dt.
\] 
\end{cor}
\begin{proof}
The upper bound for the essential norm given in Theorem \ref{thm:Upperbound} can now be written as
\[
\int_0^1  \limsup_{z\to 1} \frac{\BOAFW(z)}{\BOAFW(\phi_t(z))} \lim_{w\to1}T_t(w) dt
\]
using the second part of condition  (\ref{eq:StrongCorrelationBetween_v_and_gH}). The corollary now follows immediately from the first part of  condition (\ref{eq:StrongCorrelationBetween_v_and_gH}) and the second statement in Theorem \ref{thm:LowerboundMainH}.

\end{proof}

\section{Examples}\label{sec:Examples}
In this section we consider the standard weights, $v(z) = M_\alpha (1-\abs{z}^2)^\alpha, \ \alpha >-1$, where $M_\alpha$ is a normalization constant. Many results will, however, demand $\alpha \geq 0$ or $\alpha\in(0,1)$. 
The constant $M_\alpha$ is irrelevant for the quotient $\frac{v(z)}{v(\phi_t(z))}$, so we will use $M_\alpha=1$ in the calculations done below, although the more correct way would be to add the factor $\frac{1}{M_\alpha}$ to the function $g^{p_0}$ defined below. 

Let $\alpha\geq 0$ and $g(z)=(2(1-z))^{-\frac{\alpha}{p_0}}$ ($p_0=1$ for $H^\infty_\alpha$ and $p_0=p$ for $A^p_\alpha$). 
Clearly, $g$ is real valued on $[0,1)$. Using the Julia-Carathéodory theorem, we obtain
\begin{equation}\label{eq:CEV_PART_1}
\begin{split}
\limsup_{z\to 1}  \frac{v(z)}{v(\phi_t(z))}&= \limsup_{z\to 1} \left(  \frac{1-\abs{z}^2} {1-\abs{\phi_t(z)}^2}  \right)^{\alpha}  = \phi_t'(1) ^{-\alpha}  \\
&=  \lim_{z\to 1} \left(   \frac{2(1-z)} {2(1-\phi_t(z))}     \right)^{\alpha}       = \lim_{z\to 1}  \left(   \frac{g(\phi_t(z))} {g(z)}     \right)^{p_0 },
\end{split}
\end{equation}   
which indeed shows that the first part of condition (\ref{eq:StrongCorrelationBetween_v_and_gA}) and (\ref{eq:StrongCorrelationBetween_v_and_gH}) are satisfied for this particular pair of functions $(v,g)$. Let $\Omega := \frac{1}{2}\left(\frac{\max\{1-\alpha,1\}}{p}+ \frac{2}{p}\right)$ and notice that 
$\Omega + \frac{\alpha}{p}\in (\frac{3+\alpha}{2p},\frac{2+\alpha}{p})\subset (\frac{1}{p},\frac{2+\alpha}{p})$ when $\alpha >-1$.

\begin{lem}\label{lem:ProvingCAIA}
Let $-1<\alpha<p-2$. The function 
\[
f_c: z\mapsto \frac{g(z)}{(1-z)^c} =\frac{2^{-\frac{\alpha}{p}}}{(1-z)^{c+\frac{\alpha}{p}}}  , \ c\in \left(\Omega ,\frac{2}{p}\right) 
\]
satisfies condition (\ref{eq:ConditionApproximateidentityA}).
\end{lem}
\begin{proof}
For $1<Q_L<Q_U<\infty$, we will show that
\begin{equation}\label{eq:usefulAsymp}
\int_0^{2\pi} \frac{d\theta}{\abs{1-re^{i\theta}}^q}\asymp \frac{1}{(1-r)^{q-1}}, \ q\in[Q_L,Q_U],
\end{equation} 
where the constants involved in $\asymp$ will depend on $Q_L$ and $Q_U$, but not on $q$. By putting $Q_L= \frac{3+\alpha}{2}$ and $Q_U = 2+\alpha$ so that $cp+\alpha \in [Q_L,Q_U]$, we have

\begin{align*}
\int_{\D} \abs{\frac{2^{-\frac{\alpha}{p}}}{(1-z)^{c+\frac{\alpha}{p}}}             }^p  dA_\alpha(z)     &\asymp   \int_0^1  \frac{      (1-r^2)^\alpha 2rdr    }{(1-r^2)^{cp+\alpha-1}} = \frac{1}{2-cp},
\end{align*}
where the constants in $\asymp$ are dependent only on $\alpha$. This proves the two remaining parts of the (\ref{eq:ConditionApproximateidentityA}) condition. Let us finish the proof by showing that (\ref{eq:usefulAsymp}) holds. To this end, let $\tan \frac{\theta }{2}=  t$, which yields $\cos\theta = \frac{1-t^2}{1+t^2}$ and let $R=\frac{r^2+1}{2r}$. With these notations we have 
\begin{align*}
 \int_0^{2\pi} \frac{1}{       \abs{1-re^{i\theta}}^q         }d\theta &=\int_0^{2\pi} \frac{1}{       \left( 1+r^2-2r\cos \theta        \right)^\frac{q}{2}         }d\theta  =\frac{2}{(2r)^\frac{q}{2}}\int_0^{\pi} \frac{1}{       \left( R-\cos \theta        \right)^\frac{q}{2}         }d\theta \\
&=\frac{2}{(2r)^\frac{q}{2}}\int_0^\infty \frac{1}{       \left( R-\frac{1-t^2}{1+t^2}    \right)^\frac{q}{2}         }\frac{2dt}{1+t^2} =\frac{4}{(2r)^\frac{q}{2}}\int_0^\infty  \frac{      (1+t^2)^{\frac{q}{2}-1}  dt    }{       \left( t^2(R+1) + R-1  \right)^\frac{q}{2}         }\\
&=\frac{4}{(2r(R-1))^\frac{q}{2}}\int_0^{\infty} \frac{      (1+t^2)^{\frac{q}{2}-1}  dt    }{       \left( t^2\frac{R+1}{ R-1}  +1  \right)^\frac{q}{2}         } \\
&=\frac{4}{(2r(R-1))^\frac{q}{2}}  \sqrt{  \frac{R-1}{ R+1}   }\int_0^\infty \frac{      (1+\frac{R-1}{ R+1}t^2)^{\frac{q}{2}-1}  dt    }{       \left( t^2  +1  \right)^\frac{q}{2}         }\\
&=\frac{4}{  (1-r)^{q-1} (1+r)        }\int_0^{\infty}\frac{     \left(1+\left(\frac{1-r}{ 1+r}t\right)^2\right)^{\frac{q}{2}-1}  dt    }{       \left( t^2  +1  \right)^\frac{q}{2}         }.
\end{align*}
Moreover,
\[
  \int_0^\infty \frac{   dt   }{       \left( t^2  +1  \right)^\frac{q}{2}         }= \frac{1}{2}\int_0^1  x^\frac{q-3}{2}      (1-x)^{-\frac{1}{2}}  dx = \frac{1}{2} \beta(\frac{q-1}{2},\frac{1}{2}),
\]
where $(t^2+1)^{-1}=x$ so that $t=\sqrt{\frac{1}{x}-1}$ and $dt = - \sqrt{   \frac{x}{1-x}   }\frac{dx}{2x^2}$. Since 
\[
\int_0^{\infty}\frac{     \left(1+\left(\frac{1-r}{ 1+r}t\right)^2\right)^{\frac{q}{2}-1}  dt    }{       \left( t^2  +1  \right)^\frac{q}{2}         }
\]
is monotone with respect to $r\in[0,1]$, we obtain the inequalities
\[
 m_q:=\min\left\{  \frac{1}{2} \beta(\frac{q-1}{2},\frac{1}{2}), \frac{\pi}{2}  \right\} \leq \int_0^\infty \frac{      (1+\left(\frac{r-1}{ r+1}t\right)^2)^{\frac{q}{2}-1}  dt    }{       \left( t^2  +1  \right)^\frac{q}{2}         }\leq \max\left\{  \frac{1}{2} \beta(\frac{q-1}{2},\frac{1}{2}), \frac{\pi}{2}  \right\} =:M_q,
\]
from which it follows that
\[
2m_{Q_U} \leq   \int_0^{2\pi} \frac{    (1-r)^{q-1}     }{       \abs{1-re^{i\theta}}^q         }d\theta \leq 4M_{Q_L} .
\]

\end{proof}

\begin{rem}
In case $\alpha\geq 0$ one could show that the functions $v$ and $g$ satisfy the conditions given in Theorem \ref{thm:SomeSuffCond} and hence, deduce that condition (\ref{eq:ConditionApproximateidentityA}) holds. 
\end{rem}

Given the following lemma, which is almost trivial, we have proved that all conditions not involving $I_K$ are satisfied for the standard weights $v(z)=(1-\abs{z}^2)^\alpha$, where $0\leq \alpha<p-2$ for 
$A^p_\alpha$ and $0<\alpha<1$ for $H^\infty_\alpha$. Note that the only condition that does not allow $-1<\alpha<0$ in the Bergman case is (\ref{eq:StrongCorrelationBetween_v_and_gA}).
\begin{lem}
Let $v(z) = (1-\abs{z}^2)^\alpha$ and $g(z)= (2(1-z))^{-\alpha}, \alpha\geq 0$. Then conditions (\ref{eq:WeightCondUpperBoundWeakH}) and (\ref{eq:WeightCondUpperBoundStrong}) are satisfied. 
\end{lem}

All that remains to show is that the kernel function $K$ satisfies the demands given in the Preliminaries (Section \ref{sec:Prelim}), Section \ref{sec:UpperBound} and conditions (\ref{eq:EquivToIKBounded}) and (\ref{eq:StrongCorrelationBetween_v_and_gA}) regarding Bergman spaces and (\ref{eq:StrongCorrelationBetween_v_and_gH}) regarding weighted Banach spaces of analytic functions. Notice that conditions  (\ref{eq:StrongCorrelationBetween_v_and_gA}) and (\ref{eq:StrongCorrelationBetween_v_and_gH}) are identical regarding the kernel function part. In the next example, we will consider the Hilbert matrix operator $\HH$ on $A^p_\alpha, \ -1<\alpha<p-2,$ and on $H^\infty_\alpha , \ 0<\alpha<1$.

\begin{exam}

Let us consider the Hilbert matrix operator with the kernel $K(z,x) = \frac{1}{1-zx} $. The demands given in Section \ref{sec:Prelim} are clearly satisfied, i.e. $K$ is analytic on $\D\times \D$, $K(z,\cdot)\in H^\infty$ and $I_K:X\to X$ is a bounded operator whenever $X=A^p_\alpha,\alpha >-1, \ p>2+\alpha$ or $X=H^\infty_\alpha, \ 0 <\alpha<1$. We have 
\[
\lim_{z \to 1} T_t(z)=\lim_{z\to 1} x_z'(t)  K(z,x_z(t)) =  \lim_{z \to 1}  \left(\frac{1-z}{    (tz+1-z)^2   }  \frac{tz+1-z}{1-z}  \right)  = \frac{1}{t},
\]
which combined with (\ref{eq:CEV_PART_1}) proves the conditions (\ref{eq:StrongCorrelationBetween_v_and_gA}) and (\ref{eq:StrongCorrelationBetween_v_and_gH}). It is clear that $A^p_0 \subset A^p_\alpha$ for $\alpha \geq 0$ and using the evaluation functionals it can be seen that $H^\infty_{\alpha'} \subset A^p_0$, when $\alpha' <\frac{1}{p}$. Therefore,  condition (\ref{eq:ConditionForUpperBoundOfIKBergman}) for $\alpha\geq 0,\ p>2+\alpha$ follows from condition (\ref{eq:ConditionForUpperBoundOfIKInfty}), $0<\alpha'<1$, since we can for every $1<p<\infty$ choose $0<\alpha'<1$ such that  $\alpha' <\frac{1}{p}$ holds. Let us prove  (\ref{eq:ConditionForUpperBoundOfIKInfty}), $0<\alpha<1$. Using $T_t(z)  = w_t(z) = \frac{1}{1-(1-t)z}$, we have

\begin{align*}
 \norm{T_t\chi_{   \phi_t^{-1}(D_{\leq R,t})   }  }_{L^\infty_\alpha} &= \sup_{z\in \phi_t^{-1}(D_{\leq R,t})   } \abs{w_t(z)}(1-\abs{z}^2)^\alpha    \leq   \sup_{z\in\D } \abs{w_t(z)}(1-\abs{z}^2)^\alpha   \\
& \leq   2^\alpha \sup_{z\in\D } w_t(\abs{z})(1-\abs{z})^\alpha  =  2^\alpha \sup_{x\in[0,1)} \frac{(1-x)^\alpha}{1-(1-t)x}.
\end{align*}
If $t<1-\alpha$ the supremum is attained at
\[
x = \frac{ (1 -t -\alpha)}{(1-\alpha)(1-t)},
\]
while it otherwise is attained at $0$. Moreover, $x=0$ gives us $1$ as the supremum and $ x = \frac{ (1 -t -\alpha)}{(1-\alpha)(1-t)}$ gives the value
\begin{align*}
\frac{\left(\frac{   (1-\alpha)(1-t)  -    (1 -t -\alpha)           }{         (1-\alpha)(1-t)         }\right)^\alpha}{\left(\frac{     (1-\alpha)(1-t)     -(1-t)(1 -t -\alpha)      }{           (1-\alpha)(1-t)          }\right)}   = (1-\alpha)^{1-\alpha} (1-t)^{1-\alpha}      \frac{       ( t\alpha )^\alpha     }{     t(1-t)         }  = (1-\alpha)^{1-\alpha} \alpha ^\alpha  \frac{t^{\alpha-1}}{(1-t)^\alpha}.
\end{align*}
Combining these results we obtain
\begin{align*}
 \norm{T_t\chi_{   \phi_t^{-1}(D_{\leq R,t})   }  }_{L^\infty_\alpha} &\leq 2^\alpha \max \left\{1 ,  (1-\alpha)^{1-\alpha} \alpha ^\alpha  \frac{t^{\alpha-1}}{(1-t)^\alpha} \right\} 
\end{align*}
from which it immediately follows that (\ref{eq:ConditionForUpperBoundOfIKInfty}) holds.

Next, we prove that the conditions (\ref{eq:EquivToIKBounded}) and (\ref{eq:EssentialNormBddA}) hold, where we have added the restriction $\alpha\geq 0$ to be able to prove that (\ref{eq:EssentialNormBddA}) holds (see Remark \ref{rem:ProblemWithUnboundedWeights}). After that, we finally prove that condition (\ref{eq:EssentialNormBddH}) holds. Recall that $\Omega = \frac{1}{2}\left(\frac{\max\{1-\alpha,1\}}{p}+ \frac{2}{p}\right)$. Notice that $g(\abs{z})^{-p}\asymp v(z)$ and $T_t(z) \leq T_t(\abs{z})$. Since
\begin{align*}
\int_0^1  \sup_{c\in(\Omega,\frac{2}{p})}  \sup_{z\in\D}      \abs{   (\frac{1}{1-t}-z)^c T_t(z)  }        \sup_{w\in\D}    \abs{   \frac{g(\phi_t(w))}{g(w)}    }    dt & =
\int_0^1  \sup_{c\in( \Omega ,\frac{2}{p})}  \sup_{z\in\D}  \frac{\abs{ 1-(1-t)z    }^{c-1}}{(1-t)^c}       \sup_{w\in\D}    \abs{\frac{1-w}{1-\phi_t(w)}}^\frac{\alpha}{p}         dt \\
&\leq \int_0^1  \sup_{z\in\D}  \frac{\abs{ 1-(1-t)z    }^{\Omega-1}}{(1-t)^\frac{2}{p}}    \sup_{w\in\D}    \abs{\frac{(1-(1-t)w)}{ (1-t) }}^\frac{\alpha}{p}         dt \\
&\leq \int_0^1  \sup_{z\in\D}  \frac{\abs{ 1-(1-t)z    }^{\frac{1}{p}-1}}{ (1-t)^\frac{2+\alpha}{p} }  \max\{   t^\frac{\alpha}{p}   , (2-t)^\frac{\alpha}{p}    \}  dt \\
&\leq \int_0^1    \frac{        \max\{   t^{  \frac{1+\alpha}{p}-1 }   ,t^{\frac{1}{p}-1}  (2-t)^\frac{\alpha}{p}    \}        }{(1-t)^\frac{2+\alpha}{p}}       dt <\infty,
\end{align*}
 it follows from Remark \ref{rem:usefulRemark} that (\ref{eq:EquivToIKBounded}) and (\ref{eq:EssentialNormBddA}) are satisfied. To prove that (\ref{eq:EssentialNormBddH}) hold, we notice that $\BOAFW$ is decreasing, and hence,
\[
\frac{\BOAFW(z)}{\BOAFW(\phi_t(z))} \leq \frac{\BOAFW(z)}{\BOAFW(\phi_t(\abs{z}))} \asymp  \left(  \frac{1-\abs{z}}{1-\frac{t}{1-(1-t)\abs{z}   }}  \right)^\alpha  = \left(  \frac{1-(1-t)\abs{z}}{1-t}   \right)^\alpha.
\] 
Since $\abs{T_t(z)}\leq T_t(\abs{z}) \leq T_t(1)=\frac{1}{t}$ and $ \phi_t^{-1}(D_{>R_0,t}) \subset \D$ for all $0<R<1$, it follows that
\begin{align*}
  \int_0^1       \sup_{ z\in {B(1,\epsilon)}\cap \D }  \frac{  \abs{   T_t(z)   }   \BOAFW(z)  }{ \BOAFW(\phi_t(z)) }dt &\leq   \int_0^1       \sup_{ z\in \D }  \frac{  ((t-1)\abs{z}+1)^\alpha }{(t(1-t))^\alpha } dt \\
&\leq \int_0^1       \frac{1 }{(t(1-t))^\alpha} dt
\end{align*} 
is finite when $0<\alpha<1$.

 It now follows from Corollaries \ref{cor:EssentalNormOnA} and \ref{cor:EssentalNormOnH} for the spaces $A^p_\alpha,\ 0\leq \alpha<p-2,$ and $H^\infty_\alpha, \ 0<\alpha<1,$ respectively, that the lower bounds for the essential norms are the same as the upper bounds and the values are
\[
\norm{\HH}_{e,A^p_\alpha \to A^p_\alpha } =   \int_0^1 \left(\frac{t}{1-t}\right)^\frac{2}{p} \frac{    \phi_t'(1)   ^{-\frac{  \alpha  }{p}}  }{t}    dt =  \int_0^1  \frac{  t^{\frac{2+\alpha}{p}-1}    }{   (1-t)^\frac{2+\alpha}{p}   }    dt = \frac{\pi}{\sin\frac{(2+\alpha)\pi}{p}}
\]
and
\[
\norm{\HH}_{e,H^\infty_\alpha \to H^\infty_\alpha } =   \int_0^1\frac{    \phi_t'(1)   ^{-   \alpha }  }{t}    dt =  \int_0^1  \frac{  t^{\alpha-1}    }{   (1-t)^{\alpha} }   dt = \frac{\pi}{\sin(\alpha\pi)}.
\]
Notice that $\HH(H^0_\alpha)\subset H^0_\alpha$ by \cite[Theorem 2.1 (iii)]{AMS}, which gives the result
\[
\norm{\HH}_{e,H^0_\alpha \to H^0_\alpha } =   \frac{\pi}{\sin(\alpha\pi)}.
\]
We have also obtained the partial result for $\alpha\in(-1,0),\ p>2+\alpha$
\[
\norm{\HH}_{e,A^p_\alpha \to A^p_\alpha } \geq  \frac{\pi}{\sin\frac{(2+\alpha)\pi}{p}}.
\]
It can be shown that the expression for the upper bound given in Section \ref{sec:UpperBound} is infinite when $\alpha\in(-1,0),\ p>2+\alpha$. On the other hand, if we switch $\limsup_{z\to 1} \frac{v(z)}{v(\phi_t(z))}$ with the radial limit $\lim_{r\to 1} \frac{v(r)}{v(\phi_t(r))}$, we would attain $  \frac{\pi}{\sin\frac{(2+\alpha)\pi}{p}}$ as an upper bound. However, $\lim_{r\to 1} \frac{v(r)}{v(\phi_t(r))} = \liminf_{z\to 1} \frac{v(z)}{v(\phi_t(z))}$. The same lower bound (for the norm) was already found  by Karapetrovi\'c in \cite{BK2} using the same function $g$ except for the factor $2^{-\frac{\alpha}{p}}$. The calculations are, however, quite different, since for the specific kernel function $\frac{1}{1-zx}$ one can make use of hypergeometric functions as done in \cite{BK2}.
\end{exam}

\begin{exam}
Let $1<p<\infty$. The functions $f_c\in H^p$ given by $f_c(z)=\left(\frac{1}{(1-z)^c}\right) / \norm{ \frac{1}{(1-z)^c   }}_{H^p}, \ c<\frac{1}{p}$ converge weakly to $0$ as $c\to \frac{1}{p}$ and satisfy  
\[
\norm{\HH(f_c)  }_{H^p} \to \frac{\pi}{\sin\frac{\pi}{p}} = \norm{\HH}_{H^p\to H^p} \text{ as }  c\to \frac{1}{p}.
\]
The exact value of norm of $\HH$ on $H^p$ for $1 < p < \infty$ was established by Dostani\'c, Jevti\'c and  Vukoti\'c in \cite{DJV}.
This implies
\[
\norm{\HH}_{e,H^p\to H^p} = \norm{\HH}_{H^p\to H^p}  =  \frac{\pi}{\sin\frac{\pi}{p}}.
\]

\end{exam}

\section{Essential norm of a class of weighted composition operators}\label{sec:EssNormOfWCompOp}

We begin this section with some properties regarding the weighted composition operator acting on $A^p_v$, $H^\infty _v$ and $H^p$. The main result in this section is Theorem \ref{EssentialRaidusEqEssNorm}, which we use in Section \ref{sec:InterestingID} to obtain some interesting identities. Furthermore, we  assume that $1 \leq p <\infty$, except when stated otherwise.

In the spirit of \cite[Theorem 1]{GG} we have

\begin{thm}\label{thm:sufficientConditionForUCphiToBeCompact}
Assume that $\phi$ is an analytic self-map of $\D$ such that $\limsup_{z\to w}\abs{\phi(z)}=1$ implies $w\in\{a_j\in \partial \D, j=1,2,\ldots ,k\}$. Assume furthermore that $\psi\in H^\infty$ is continuous on  $\D\cup\left(\closed{\D}\cap \bigcup_{j=1}^k B(a_j,\rho)\right)  $ for some $\rho>0$ and that $\psi(a_j)=0$ for $j=1,2,\ldots , k$. 
\\
If $X=A^p_v$, $X=H^p$ or  $X=H^\infty_v$ and  $\psi C_\phi:X \to X$ is a bounded   weighted composition operator
, then  $\psi C_\phi$ is compact. 
\end{thm}
\begin{proof}
Let
\begin{equation}\label{almostContraction}
R:=\sup_{z\in\D\setminus \left( \bigcup_{j=1}^k B(a_j,\rho) \right) } \abs{\phi(z)}.
\end{equation}
Since $\limsup_{z\to w}\abs{\phi(z)}=1$ implies $w\in\{a_j\in \partial \D, j=1,2,\ldots ,k\}$ we have $R<1$. We may assume that $\psi\neq 0$ since the null operator is trivially compact. 

Let $(f_n)_n\subset X$ be a bounded sequence such that $ f_n \to 0$  uniformly on compact subsets of $\D$  as $n\to\infty$. Let $\epsilon>0$ and choose $N$ such that for $n>N$ it holds that
\[
\sup_{z\in \closed{B(0,R)}} \abs{f_n(z)} \leq \frac{\epsilon}{4\pi \norm{\psi}_\infty}.
\]
 If $\norm{\psi C_\phi (f_n)}_X\to 0$ as $n\to \infty$, then $\psi C_\phi$ is compact by \cite[Lemma 3.3]{CPPR}. Let $M=\norm{C_\phi}_{X\to X} \sup_n \norm{f_n}_X$. Since $\psi$ is zero on $\{a_j:j=1,2,\ldots ,k\}$ and continuous on \[\D\cup(\closed{\D}\cap \bigcup_{j=1}^k B(a_j,\rho)),\] there exists $\rho_0>0$ such that $\abs{\psi(z)}<\frac{\epsilon}{2M}$ when $z\in U_0 := \D\cap\bigcup_{j=1}^k  B(a_j,\rho_0)$. Let $U_c = \D\setminus U_0$. We have 
\[
\phi(r\D\setminus U_0)\subset \phi(U_c)\subset \closed{B(0,R)} \subset \D, \ r\in(0,1),
\]
 where $R$ is given in (\ref{almostContraction}). 
\\\\
Assume first that $X=H^p$. Then
\begin{align*}
2\pi \norm{\psi C_\phi(f_n)}_{H^p}^p &= \sup_{r<1}\int_0^{2\pi} \abs{\psi C_\phi(f_n)(re^{it})}^p  \abs{dt} = \sup_{r<1}\int_{r \partial \D}  \abs{\psi C_\phi(f_n)(z)}^p  \frac{|dz| }{r}.
\end{align*}
Therefore, assuming $n>N$ we have
\begin{align*}
\int_{r \partial \D}  \abs{\psi C_\phi(f_n)(z)}^p  \abs{dz}  &=\int_{r \partial \D\cap U_0}  \abs{\psi(z)C_\phi(f_n)(z)}^p   \abs{dz} + \int_{r \partial \D\cap U_c}  \abs{\psi(z)C_\phi(f_n)(z)}^p  \abs{dz} \\
&\leq \left(\frac{\epsilon}{2M}\right) ^p\int_{r \partial \D}  \abs{C_\phi(f_n)(z)}^p \abs{dz} + \norm{\psi}_\infty^p \int_{r \partial \D\cap U_c}  \abs{(f_n\circ \phi)(z)}^p   \abs{dz} \\
&\leq \left(\frac{\epsilon}{2M}\right)^p 2\pi r\norm{C_\phi (f_n)}^p_{H^p}+ \norm{\psi}_\infty^p \sup_{z\in U_c} \abs{(f_n\circ \phi)(z)}^p  \int_{r \partial \D\cap U_c}   \abs{dz}   \\
&\leq \left(\frac{\epsilon}{2}\right)^p2\pi r+ \norm{\psi}_\infty^p 2\pi r\sup_{z\in \closed{B(0,R)}} \abs{f_n(z)}^p <2\pi r\epsilon^p
\end{align*}
and the result follows. Assume now that $X=A^p_v$. Again for $n>N$, we have
\begin{align*}
\int_{\D}  \abs{\psi C_\phi(f_n)(z)}^p  dA_v(z)  &=\int_{U_0}  \abs{\psi C_\phi(f_n)(z)}^p  dA_v(z) + \int_{U_c}  \abs{\psi C_\phi(f_n)(z)}^p  dA_v(z) \\
&\leq \left(\frac{\epsilon}{2M}\right) ^p\int_{\D}  \abs{C_\phi(f_n)(z)}^p  dA_v(z) + \norm{\psi}_\infty^p \int_{U_c}  \abs{(f_n\circ \phi)(z)}^p  dA_v(z) \\
&\leq \left(\frac{\epsilon}{2M}\right)^p \norm{C_\phi (f_n)}^p_{A^p_v}+ \norm{\psi}_\infty^p \sup_{z\in U_c} \abs{(f_n\circ \phi)(z)}^p  \int_{U_c}   dA_v(z)   \\
&\leq \left(\frac{\epsilon}{2}\right)^p+ \norm{\psi}_\infty^p \sup_{z\in \closed{B(0,R)}} \abs{f_n(z)}^p <\epsilon^p.
\end{align*}
If $X=H^\infty_v$, then  we have

\begin{align*}
\norm{\psi C_\phi(f)}_{H^\infty_v} &= \sup_{z\in\D}\abs{\psi C_\phi(f_n)(z)v(z)} \leq  \sup_{z\in U_0}\abs{\psi C_\phi(f_n)(z)v(z)}  + \sup_{z\in  U_c}\abs{\psi C_\phi(f_n)(z)v(z)} \\
&\leq  \frac{\epsilon}{2M}\sup_{z\in\D}\abs{C_\phi(f_n)(z)v(z)}  + \norm{\psi}_\infty\sup_{z\in\mathbb U_c}\abs{(f_n\circ\phi)(z)}\norm{v}_\infty \\
&\leq  \frac{\epsilon}{2}  + \norm{\psi}_\infty\sup_{z\in \closed{B(0,R)}}\abs{f_n(z)} <\epsilon.
\end{align*}

\end{proof}

\begin{lem}\label{lem:UpperBoundForEssRadi}
Assume that $\phi$ is a univalent self-map of $\D$ such that $\limsup_{z\to w}\abs{\phi(z)}=1$ is equivalent to $w=a\in\partial \D$ and that the angular derivative $\phi'(z)$ exists at $a$. Assume furthermore that both $\phi'$ and $\psi\in H^\infty$  are continuous on  $\D\cup\left(\closed{\D}\cap  B(a,\rho)\right)  $ for some $\rho>0$. 
If $X=A^p_\alpha, \ \alpha>-1$ with $\ s=\frac{2+\alpha}{p}$, 
then  
\[
\norm{\psi C_\phi}_{e,X\to X} \leq    \abs{         \frac{\psi(a)}{     \phi' (a)^s}             }.
\]
The same inequality holds for $X=H^p$ with $ s=\frac{1}{p}$, under the additional assumption that \\ $\lim_{r\to 1}\int_{\phi(r\partial \D)} \abs{   f(z)   }^p  \abs{  dz  } \leq 2\pi\norm{f}^p_{H^p}, \ f\in H^p$. 
\end{lem}
\begin{proof}
Since $\phi$ is univalent,  $\abs{\phi'(z)}>0$ for all $z\in\D$. Let $s\geq 0$ and $\eta= \left(  \frac{\psi}{                    ({\phi'})^s                } - \frac{\psi(a)}{              {\phi'(a)}^s               }   \right)({\phi'})^s $. Then $\eta C_\phi$ is compact according to Theorem \ref{thm:sufficientConditionForUCphiToBeCompact}. Therefore, we have
\begin{equation}\label{eq:EssRadiusOfWeightedCompOpBound}
\norm{\psi C_\phi}_{e,X\to X} = \norm{           \frac{        \psi(a){(\phi')}^s         }{             \phi'(a)^s        } C_\phi+\eta C_\phi         }_{e,X\to X} = \norm{           \frac{        \psi(a){(\phi')}^s         }{             \phi'(a)^s        } C_\phi     }_{e,X\to X} =  \abs{         \frac{\psi(a)}{     \phi' (a)^s}             }  \norm{   {(\phi')}^s C_\phi    }_{e,X\to X}.
\end{equation}
Moreover, for $X=A^p_\alpha$, we choose $s=\frac{2+\alpha}{p}$ to obtain
\begin{align*}
\pi \norm{(\phi')^s C_\phi (f) }_{ A^p_\alpha }^p &= \int_\D    \abs{        {\phi'(z)}^s C_\phi (f)(z)          }^p       (1+\alpha)(1-\abs{z}^2)^\alpha  dA(z)\\
&= \int_\D    \abs{ f\circ \phi(z)}^p      \abs{\phi'(z)}^\alpha (1+\alpha)(1-\abs{z}^2)^\alpha  \abs{\phi'(z)}^2 dA(z)\\
&\leq \int_{\D}    \abs{ f\circ \phi(z)}^p      (1+\alpha)(1-\abs{\phi(z)}^2)^\alpha  \abs{\phi'(z)}^2 dA(z)\\
&= \int_{\phi(\D)}    \abs{ f(w)}^p      (1+\alpha)(1-\abs{w}^2)^\alpha  dA(w)\leq  \pi\norm{f }_{A^p_\alpha}^p,
\end{align*}
where the first inequality is obtained  by the Schwarz–Pick lemma. The variable substitution, $w=\phi(z)$, holds due to $\phi$ being univalent. Therefore,
\[
\norm{(\phi')^s C_\phi}_{A_\alpha^p  \to A_\alpha^p} \leq 1.
\]

For $X=H^p$ and $s = \frac{1}{p}$ we have
\begin{align*}
2\pi \norm{(\phi')^s C_\phi(f)}_{H^p}^p &= \sup_{r<1}\int_0^{2\pi} \abs{(\phi')^s C_\phi(f)(re^{it})}^p \abs{dt} = \sup_{r<1}\int_{r \partial \D}  \abs{(\phi')^s C_\phi(f)(z)}^p  \frac{|dz|}{r}.
\end{align*}
Therefore, using the additional assumption, we obtain
\begin{align*}
\int_{r \partial \D}  \abs{\phi'(z)^s C_\phi(f)(z)}^p  \abs{dz}  & =  \int_{r \partial \D}  \abs{ f(\phi(z))}^p  \abs{d\phi(z)}  =  \int_{\phi(r \partial \D)}  \abs{ f(w)}^p  \abs{dw} \leq   2\pi \norm{f}_{H^p}^p,
\end{align*}
and the   statement  follows as above.

\end{proof}

\begin{rem}

In this remark we give you some examples when $\lim_{r\to 1}\int_{\phi(r\partial \D)} \abs{   f(z)   }^p  \abs{  dz  } \leq 2\pi\norm{f}^p_{H^p}$ does not hold for every $f$. Let $\phi$ be a Riemann map that maps $\D$ onto a set $S$, where $S$ is a simply connected, nonconvex region, subset of $\D$ such that the length  $L(\partial S)>L(\partial \D)$. Then $\phi(\partial \D) = \partial  S$. It is now easy to see that, for $f\equiv1$, we have $2\pi \norm{f}^p_{H^p} = 2\pi < L(\partial S) = \lim_{r\to 1} L(\phi(r \partial \D)) = \lim_{r\to 1} \int_{\phi(r \partial \D)}  \abs{ f(w)}^p \frac{ \abs{dw}     }{r}$. 
Furthermore, to assume that $\phi(r\D)$ is convex for $r<1$ close to $1$ does not grant 
\[
\int_{\phi(r \partial \D)}  \abs{ f(w)}^p  \abs{dw} \leq   2\pi \norm{f}_{H^p}^p,
\]
although $L(\partial S)<L(\partial D)$. As an example, consider one of the curves 
\[
S_\theta = \{z=e^{it}:t\in \left(\theta,2\pi - \theta   \right) \}\cup \{\cos \theta +iy: y \in [-\sin \theta, \sin \theta ]  \}, \ \theta\in \left(0,  \frac{\pi}{2}\right)
\] 
and the function $(1-z)^\frac{2}{p}$.

\end{rem}


\begin{thm}\label{EssentialRaidusEqEssNorm}
Assume that $\phi$ is a self-map of $\D$ with a fixed point $a\in \partial\D$ such that $\limsup_{z\to w}\abs{\phi(z)}=1$ is equivalent to $w=a$ and that the angular derivative $\phi'(z)$ exists at $a$. Assume furthermore that both $\phi'$ and $\psi\in H^\infty$ are continuous on  $\D\cup\left(\closed{\D}\cap  B(a,\rho)\right)  $ for some 
$\rho>0$. 
\\
If $X=A^p_\alpha, \ \alpha>-1$ with $\ s=\frac{2+\alpha}{p}$ and $\phi$ is univalent or $X=H^\infty_\alpha, \ \alpha> 0$ with $s=\alpha$, 
we have
\[
 \norm{\psi C_\phi}_{e, X \to X} =            \frac{           \abs{ \psi(a)  }       }{     \phi' (a)^s} = \limsup_{ z \to a} \frac{          \abs{\psi(z) }      (  1-\abs{z}^2  )^s  }{     (   1-\abs{\phi(z)}^2  )^s     } = \limsup_{\abs{\phi(z)}\to 1} \frac{          \abs{\psi(z) }      (  1-\abs{z}^2  )^s  }{     (   1-\abs{\phi(z)}^2  )^s     } .
\]
\end{thm}

\begin{proof}
The last equality in the statement follows from the assumption that $\limsup_{z\to w}\abs{\phi(z)}=1$ if and only if $w=a$. Combining the assumptions with the Julia-Carathéodory Theorem \cite[Theorem 2.44]{CM} we have
\begin{align*}
 \limsup_{z\to a} \frac{          \abs{\psi(z) }      (  1-\abs{z}^2  )^s  }{     (   1-\abs{\phi(z)}^2  )^s     }   = \abs{\psi(a)} \left( \frac{  1+\abs{a} }{1+\abs{\phi(a)}      } \right)^s     \left( \limsup_{z\to a}  \frac{  1-\abs{z}    }{        1-\abs{\phi(z)}       }   \right)^s = \frac{      \abs{\psi(a)}     }{       \abs{\phi'(a)}^s   },
\end{align*}
which is the second equality in the statement. To obtain the essential norm we first consider the space $A^p_\alpha$. By combining the following standard approach (see  Proposition 3.6 in \cite{GLW})
\begin{align*}
\norm{\psi C_\phi}_{e, A^p_\alpha \to A^p_\alpha} &\geq \limsup_{\abs{z}\to 1} \frac{   \abs{\psi(z)}        \norm{\delta_{\phi(z)}     }_{(A^p_\alpha)^*}         }{           \norm{\delta_z}_{(A^p_\alpha)^*}          }  
 = \limsup_{\abs{z}\to 1} \frac{    \abs{ \psi(z)  }  (  1-\abs{z}^2  )^s  }{      (  1-\abs{\phi(z)}^2   )^s    } \\
&\geq  \limsup_{z\to a} \frac{          \abs{\psi(z) }      (  1-\abs{z}^2  )^s  }{     (   1-\abs{\phi(z)}^2  )^s     }  \\
\end{align*}
with Lemma  \ref{lem:UpperBoundForEssRadi} we obtain the first equality in the statement. Concerning $X=H^\infty_\alpha$, the result
\[
\norm{\psi C_\phi}_{e,H^\infty_\alpha \to H^\infty_\alpha}  = \limsup_{\abs{\phi(z)}\to 1} \frac{    \abs{ \psi(z)  }  (  1-\abs{z}^2  )^s  }{      (  1-\abs{\phi(z)}^2   )^s    }
\]
has  been proved in \cite[Theorem 2.1]{MR} by Montes-Rodriguez.  
\end{proof}

\begin{rem}
It is easy to see that
\[
\norm{\psi C_\phi}_{e, X \to X}\geq \max_{1\leq j\leq k}\{     \frac{      \abs{\psi(a_j)}     }{       \abs{\phi'(a_j)}^s   }       \}
\]
if there are a finite number of points $a_1,\ldots, a_k$ satisfying the assumptions of $a$ in Theorem \ref{EssentialRaidusEqEssNorm} except not necessarily being fixed points for $\phi$. 
\end{rem}

\section{Behavior of norm and essential norm of the Hilbert matrix operator}\label{sec:InterestingID}

As a consequence of our earlier results, we begin this section by pointing out an interesting identity that we can represent the essential norm of the operator $I_K$ on $A^p_\alpha$ and $H^\infty_\alpha$ as an integral average of essential norms of certain weighted composition operators. Moreover, we prove that this kind of equality does not hold for norms in general, using $\HH:A^p\to A^p,  p\in(2,2.7)$ as a counterexample. We also determine the exact norm of the Hilbert matrix operator on $H^\infty_{w_{\alpha}}$ for the  weight $w_\alpha(z) =(1-|z|)^\alpha$, which also coincides with its essential norm. Let us also point out that the previous results in this article concerning the weighted Bergman spaces and the weighted Banach spaces of analytic functions also hold for  the  weights   $w_\alpha$, where $ 0\leq \alpha <p-2<\infty$ and $0<\alpha<1 $, respectively.

 The respective main results are theorems \ref{remarkable} and \ref{weightedspaceresult}. We finish the paper with presenting a new approach on the problem of determining $\norm{\HH}_{A^p_\alpha\to A^p_\alpha}, \ p>2+\alpha \geq 2$.

First, let us  assume that all conditions regarding the function $K$ are satisfied. We also  assume that $T_t\in H^\infty, \ t\in(0,1)$. If
\[
\sup_{w\in\D} \sup_{z\in \D \setminus  B(1,\epsilon)  }\abs{K(z,w)}<\infty
\]
holds, then it follows from (\ref{eq:StrongCorrelationBetween_v_and_gA}) that $T_t\in H^\infty$. From Theorem \ref{EssentialRaidusEqEssNorm}  we get that 
\[
\norm{T_tC_{\phi_t}}_{e, A^p_\alpha \to A^p_\alpha} =  \lim_{z\to 1 }     T_t(z) \frac{t^s}{(1-t)^s}, \ s=\frac{2+\alpha}{p}.
\]
Now, we can use Corollary \ref{cor:EssentalNormOnA} to obtain  the remarkable result
\[
\norm{  \int_0^1 T_tC_{\phi_t}dt   }_{e, A^p_\alpha \to A^p_\alpha}    = \norm{I_K}_{e, A^p_\alpha \to A^p_\alpha} = \int_0^1 \norm{T_tC_{\phi_t}}_{e, A^p_\alpha \to A^p_\alpha} dt.
\]
Similarly, we obtain from Corollary \ref{cor:EssentalNormOnH} the corresponding result for $H^\infty_\alpha, \ 0<\alpha<1$,
\[
\norm{  \int_0^1 T_tC_{\phi_t}dt   }_{e, H^\infty_\alpha \to H^\infty_\alpha}    = \norm{I_K}_{e, H^\infty_\alpha \to H^\infty_\alpha} = \int_0^1 \norm{T_tC_{\phi_t}}_{e, H^\infty_\alpha \to H^\infty_\alpha} dt.
\]

In the case  of the Hilbert matrix operator on $A^p_\alpha$ we formulate this result as a part of the following interesting  theorem.

\begin{thm}\label{remarkable}   For the Hilbert matrix operator $\mathcal H: A^p_\alpha \to A^p_\alpha$, $p>2+\alpha\geq 2$, 
\[
\norm{\HH}_{e,  A^p_\alpha \to A^p_\alpha }=\int_0^1 \norm{w_t C_{\phi_t}}_{e, A^p_\alpha\to A^p_\alpha}dt 
= \frac{\pi}{\sin\frac{(2+\alpha)\pi}{p}}.
\]
On the contrary, the norm 
$\norm{\HH}_{ A^p \to A^p}$ is strictly less than $\int_0^1 \norm{w_t C_{\phi_t}}_{A^p\to A^p} dt$ for some $p>2$.
\end{thm}
\begin{proof}
The first statement follows from the discussion above and the observation 
\begin{equation}\label{eq:EssNormOfHilbKern}
\norm{w_tC_{\phi_t}}_{e, A^p_\alpha \to A^p_\alpha} =  \frac{t^{\frac{2+\alpha}{p}-1}}{(1-t)^{\frac{2+\alpha}{p}}}. 
\end{equation}
It remains to prove the second statement of  Theorem  \ref{remarkable}. More precisely, we will show that regarding the Hilbert matrix operator $\mathcal H$  on $A^p$ we have
\[
\norm{w_t C_{\phi_t}}_{A^p \to A^p} > \frac{t^{\frac{2}{p}-1}  }{(1-t)^{\frac{2}{p}}}
\]
when $p$ is small enough $(<2.7)$ for $t \in (0,t_p)$ for some $t_p\in(0,1)$. After that, it follows from (\ref{eq:EssNormOfHilbKern}) that 
\[
\int_{t_p}^1\norm{w_tC_{\phi_t}}_{A^p_\alpha \to A^p_\alpha} dt \geq  \int_{t_p}^1 \frac{t^{\frac{2}{p}-1}}{(1-t)^{\frac{2}{p}} } dt
\] 
and hence, we obtain the strict inequality
\[
\norm{\HH}_{A^p} = \int_0^1 \frac{t^{\frac{2}{p}-1}  }{(1-t)^{ \frac{2}{p} }} dt  < \int_0^1 \norm{w_tC_{\phi_t}}_{A^p} dt.
\]
The upper bound on $p$ is obtained as follows.  We have 
\begin{align*}
\pi \norm{w_t C_{\phi_t}(f)}_{A^p} ^p &=  \int_\D \abs{w_t(z)}^p \abs{C_{\phi_t} (f)(z)}^p  dA(z) = \int_{\phi_t(\D)}  \frac{\abs{z}^p}{t^p} \abs{f(z)}^p \abs{({\phi_t^{-1}})'(z)}^2 dA(z)\\
&= \int_{\phi_t(\D)}  \frac{t^{2-p}}{(1-t)^2} \abs{z}^{p-4}\abs{f(z)}^p dA(z).
\end{align*}
Let us define
\[
\B_t :A^p \to \mathbb R, f\mapsto  \int_{\phi_t(\D)} \abs{z}^{p-4}\abs{f(z)}^p \frac{dA(z)}{\pi} , \ t\in (0,1), 
\]
and notice that for a fixed $f$ the function $\B_t(f)$ is non-increasing with respect to $t$ since   
\[
\phi_{t_2}(\D)\subset \phi_{t_1}(\D) ,  \ 0<t_1<t_2<1.
\] 
From the equality $\bigcup_{0<t<1} \phi_t(\D) = B(\frac{1}{2},\frac{1}{2})$ we now obtain 
\begin{align*}
\sup_{t\in(0,1)}\B_t(1) &= \int_{  B(\frac{1}{2},\frac{1}{2})  } \abs{z}^{p-4} \frac{dA(z)}{\pi}  = \int_0^1  \int\limits_{- \arccos r }^{     \arccos r }   r^{1+p-4}    d\theta \frac{dr}{\pi}\\
&= \left[   2\arccos r  \int_0^r s^{p-3}   \frac{ds}{\pi} \right]_{r=0}^1 + \frac{1}{p-2}\int_0^1   \frac{2}{\sqrt{   1-r^2   }}  r^{p-2} \frac{dr}{\pi}  =  \frac{ \beta (\frac{p-1}{2}, \frac{1}{2}) }{\pi(p-2)}.
\end{align*}
The zero, denoted by $p_0$, of the strictly decreasing function $h(p):= \frac{ \beta (\frac{p-1}{2}, \frac{1}{2}) }{\pi(p-2)} -1, \ p\in(2,\infty)$ can be  obtained numerically, $p_0\approx 2.703$ and for $p\in(2,p_0)$ we can now find a $t_p\in(0,1)$ such that
\[
\B_{t_p}(1)>1.
\]
We can conclude that $\B_{t}(1)>1$ holds true for $t\in(0,t_p)$, because $\B_t(1)$ is non-increasing. To finish the proof we observe that for $p\in(2,p_0)$ and $t\in(0,t_p)$ we have
\[
\frac{t^{\frac{2}{p}-1}}{(1-t)^{\frac{2}{p}} } < \frac{t^{\frac{2}{p}-1}}{(1-t)^{\frac{2}{p}} }\B_{t}(1)^{\frac{1}{p}} \leq \sup_{f\in B_{A^p}} \left(  \frac{t^{2-p}}{(1-t)^2 }\B_{t}(f) \right)^{\frac{1}{p}} = \sup_{f\in B_{A^p}} \norm{w_t C_{\phi_t}(f)}_{A^p}  =\norm{w_t C_{\phi_t}}_{A^p \to A^p}.
\]

\end{proof}

In  \cite{LMW} the authors  determined the exact  norm $\norm{\mathcal H}_{H^\infty_\alpha\to H^\infty_\alpha} = \pi/\sin(\alpha \pi)$,  when $ 0 < \alpha \le \frac{2}{3}$ with the weight $\BOAFW(z) = (1 - \abs{z}^2)^\alpha$. In the case $ \frac{2}{3}< \alpha <1$ a worse upper bound was obtained. Next, we show that $\norm{\mathcal H}_{H^\infty_{w_\alpha}\to H^\infty_{w_\alpha}} = \pi/\sin(\alpha \pi), 0 < \alpha < 1 $, where $w_\alpha$ is the equivalent weight  $w_\alpha(z) = (1 - \abs{z})^\alpha$. 
\begin{thm}\label{weightedspaceresult}  Let $0 < \alpha <1$ and consider the Hilbert matrix operator $\mathcal H: H^\infty_{w_\alpha}  \to H^\infty_{w_\alpha}$.
Then  
\[
\norm{\mathcal H}_{ H^\infty_{w_\alpha}  \to  H^\infty_{w_\alpha}} =\norm{\mathcal H}_{e, H^\infty_{w_\alpha}  \to  H^\infty_{w_\alpha}}= \frac{\pi}{\sin(\alpha \pi)}.
\]
\end{thm}
\begin{proof}
Let $0<\alpha<1$ and $g(z) = (1-z)^{-\alpha}$. Similarily to Section \ref{sec:Examples} it can be shown that conditions (\ref{eq:WeightCondUpperBoundWeakH}), (\ref{eq:WeightCondUpperBoundStrong}) and (\ref{eq:StrongCorrelationBetween_v_and_gH}) are satisfied. Since the norms $\norm{\cdot}_{H_{w_\alpha}^\infty}$ and $\norm{\cdot}_{H_\alpha^\infty}$ are equivalent, it is clear that (\ref{eq:EssentialNormBddH}) and (\ref{eq:ConditionForUpperBoundOfIKInfty}) also hold. Using Corollary \ref{cor:EssentalNormOnH} we now obtain
\[
 \norm{\mathcal H}_{H^\infty_{w_\alpha}  \to  H^\infty_{w_\alpha}}  \geq  \norm{\mathcal H}_{e, H^\infty_{w_\alpha}  \to  H^\infty_{w_\alpha}} = \frac{\pi}{\sin(\alpha \pi)}. 
\]

However, for $\norm{f}_{ H^\infty_{w_\alpha}}\le 1,$ 
\begin{align*}
\norm{w_tC_{\phi_t}(f)}_{ H^\infty_{w_\alpha}}&= \sup_{z\in \D} \abs{w_t(z)} |f(\phi_t(z))| ( 1- |\phi_t(z)|)^\alpha  \Big(  \frac{1 - |z|}{ 1- |\phi_t(z)|  }   \Big)^\alpha
\le  \sup_{z\in \D} \abs{w_t(z)}  \Big(  \frac{1 - |z|}{ 1- |\phi_t(z)|  }   \Big)^\alpha \\
& \leq  \sup_{z\in \D} \Big(  \frac{1 }{ 1- (1-t)|z|}    \Big)^{1-\alpha} \Big(  \frac{1 - |z|}{ 1- (1-t)|z| -t }   \Big)^\alpha
=\frac{t^{\alpha-1}}{(1-t)^\alpha}.
\end{align*}
Therefore,
\[
\norm{\mathcal H}_{ H^\infty_{w_\alpha}  \to  H^\infty_{w_\alpha}}\le \int_0^1 \norm{w_t C_{\phi_t}}_{ H^\infty_{w_\alpha}}
=\int_0^1 \frac{t^{\alpha-1}}{(1-t)^\alpha} \ dt=\frac{\pi}{\sin(\alpha \pi)}.
\]

\end{proof}

It is still an open problem to determine the exact norm of  $\mathcal H$  on the weighted Bergman spaces for the standard weights when 
$2+\alpha < p < 2+\alpha +\sqrt{(2 + \alpha)^2-(\sqrt {2} - \frac{1}{2})(2 +\alpha)}.$ Therefore, it is of interest to  know when the norm and the essential norm of operators coincide  for such spaces. We conclude the article with a result stating that the pair $(A^p_v,A^q_v)$ has the weak maximizing property when $q\geq p >1$.

\begin{thm}\label{thm:WMP}
Let $T\in \mathcal L (A^p_v , A^q_v),\  q\geq p >1$
. Then at  least one of the following holds: 
\begin{enumerate}
\item  There exists a sequence $(f_n)_n$ of unit vectors such that $f_n\to 0$ weakly (equivalently, in the topology of compact convergence) and
\[
\lim_{n\to \infty}\norm{   T(f_n)    }_{A^q_v} = \norm{T}_{  A^p_v \to A^q_v   }.
\]

\item There exists a function $f\in B_{A^p_v}$ such that
\[
\norm{   T(f)    }_{A^q_v} = \norm{T}_{  A^p_v \to A^q_v   }.
\]
\end{enumerate}
\end{thm}
\begin{proof}
The theorem holds if $\norm{T}_{  A^p_v \to A^q_v   }=0$. Assume now that $\norm{T}_{  A^p_v \to A^q_v   }>0$. Then $\norm{T}_{  A^p_v \to A^q_v   }= \lim_{n\to\infty}\norm{T(f_n)}_{A^q_v}$ for some sequence $(f_n)_n\subset A^p_v$ with $\norm{f_n}_{A^p_v}=1$ for all $n$. Since the closed unit ball $B_{A^p_v}, \ p>1$ of a reflexive space is weakly compact we may, by going to a subsequence, assume that the  sequence $(f_n)_n$ converges weakly   to some $f\in B_{A^p_v}$. Assume that $f\not\equiv 0$. Since $f_n-f\to 0$ weakly we have 
$T(f_n)\to T(f)$ weakly. Take $\epsilon>0$ and by Lemma \ref{lem:AnLpIdentity}, choose $N$ such that $n>N$ implies 
\[
 1- \norm{f_n-f}^p_{A^p_v} =  \norm{f_n}^p_{A^p_v} - \norm{f_n-f}^p_{A^p_v} \geq \norm{f}^p_{A^p_v} - \epsilon
\]
and
\[
  \norm{T(f_n)}^q_{A^q_v} - \norm{T(f_n)-T(f)}^q_{A^q_v} \leq  \norm{T(f)}^q_{A^q_v} + \epsilon.
\]
The rest of the proof will follow from similar computations as in \cite{PT}, which yield that the normalized weak limit $\frac{f}{\norm{f}_{A^p_v}}$ is a maximizing function for $T$.
\end{proof}


\begin{cor}
Let $p\geq \alpha+2\geq 2$ and consider $\HH:A^p_\alpha \to A^p_\alpha$. If $\HH$ does not attain its norm, that is, $\norm{\HH f}_{A^p_\alpha}<\norm{ \HH }_{  A^p_\alpha \to A^p_\alpha   }$ for all $f\in B_{A^p_\alpha}$. Then
$\norm{\HH}_{  A^p_\alpha \to A^p_\alpha   }=  \norm{\HH}_{e, A^p_\alpha \to A^p_\alpha   } = \frac{\pi}{\sin\frac{(2+\alpha)\pi}{p}} $.
\end{cor}
\begin{proof}
According to Theorem \ref{thm:WMP} we can assume that $(f_n)_n$ is a normalized maximizing weak null sequence for $\HH$. For a compact operator $L\in \mathcal L (A^p_\alpha , A^p_\alpha)$ we have
\[
\norm{\HH (f_n)}_{A^p_\alpha} \leq  \norm{ (\HH - L) (f_n)}_{A^p_\alpha} + \norm{ L (f_n)}_{A^p_\alpha} \leq \norm{\HH - L}_{  A^p_\alpha \to A^p_\alpha   } + \norm{ L (f_n)}_{A^p_\alpha}    
\]
Letting $n\to \infty$ we obtain $\norm{\HH }_{  A^p_\alpha \to A^p_\alpha   } \leq  \norm{\HH - L}_{  A^p_\alpha \to A^p_\alpha   } $ for all $L$, since compact operators transform weak null sequences into norm null sequences. Taking the infimum over compact operators $L$ we obtain the statement.
\end{proof}

\begin{rem}
It is clear that $\norm{T}_{A^p_v \to A^q_v   } = \norm{T}_{e, A^p_v \to A^q_v }$ holds also in the more general case when $T\in \mathcal L (A^p_v , A^q_v),\  q\geq p >1$ does not attain its norm. In particular, this result may also be useful in determining the norm of $\HH$ on $A^p_\alpha, p>2+\alpha$ when $-1<\alpha<0$.
\end{rem}

\subsection*{Acknowledgments} The first two authors were  partly supported by the Academy of Finland project 296718 and the second author also received support from the Emil Aaltonen Foundation.  The third author  is grateful for the financial support from the Doctoral Network in Information Technologies and Mathematics at \AA bo Akademi University.

\end{document}